\documentclass[12pt]{article}
\usepackage{amssymb,amsthm,hyperref}
\newtheorem{thm}{Theorem}[section]
 
 \newtheorem{cor}{Corollary}[section]
 \newtheorem{lem}{Lemma}[section]
 \newtheorem{prop}{Proposition}[section]
 \newtheorem{defn}{Definition}[section]
\theoremstyle{remark}
\newtheorem{rem}{Remark}[section]

\title{Relaxation Limit in Besov Spaces for Compressible Euler
Equations}
\author{Jiang Xu\thanks{E-mail: jiangxu\underline{ }79@yahoo.com.cn}, \ Zejun Wang\thanks{E-mail: wangzejun@gmail.com}\\
\small{\textit{Department of Mathematics}},
\\ \small{\textit{Nanjing
University of Aeronautics and Astronautics}}, \\
\small{\textit{Nanjing 211106, P.R.China}}\\[5mm]
}
\date{}
\begin{document}
\maketitle{}
\begin{abstract}
The relaxation limit in critical Besov spaces for the
multidimensional compressible Euler equations is considered. As the
first step of this justification, the uniform (global) classical
solutions to the Cauchy problem with initial data close to an
equilibrium state are constructed in the Chemin-Lerner's spaces with
critical regularity. Furthermore, it is shown that the density
converges towards the solution to the porous medium equation, as the
relaxation time tends to zero. Several important estimates are
achieved, including a crucial estimate of commutator.
\end{abstract}

\hspace{-0.5cm}\textbf{Keywords:}
\small{compressible Euler equations, classical solutions, relaxation limit, Chemin-Lerner's spaces}\\

\hspace{-0.5cm}\textbf{AMS subject classification:} \small{35L25,\
35L45,\ 76N15}

\section{Introduction and Main Results}
In a suitable nondimensional form, the multidimensional compressible
Euler equations for a polytropic fluid read as
\begin{equation}
\left\{
\begin{array}{l}
\partial_{t}\rho + \nabla\cdot(\rho\textbf{v}) = 0 , \\
\partial_{t}(\rho\textbf{v}) +\nabla\cdot(\rho\textbf{v}\otimes\textbf{v}) +
\nabla P =-\frac{\rho\textbf{v}}{\tau} .
\end{array} \right.\label{R-E1}
\end{equation}
Here $\rho = \rho(t, x)$ is the fluid density function of
$(t,x)\in[0,+\infty)\times\mathbb{R}^{d}$ with $d\geq2$;
$\textbf{v}=(v^1,v^2,\cdot\cdot\cdot,v^{d})^{\top}$($\top$
represents the transpose) denotes the fluid velocity. The pressure
$P=P(\rho)$ satisfies the usual $\gamma$-law:
$$
P(\rho)=A\rho^{\gamma}(\gamma\geq 1)
$$
where $A>0$ is some physical constant, the adiabatic exponent
$\gamma>1$ corresponds to the isentropic flow and $\gamma=1$
corresponds to the isothermal flow; $0<\tau\leq1$ is a (small)
relaxation time. The notation $\nabla,\otimes$ are the gradient
operator and the symbol for the tensor products of two vectors,
respectively.

In this paper, we are going to study the Cauchy problem of the
compressible Euler equations (\ref{R-E1}) subject to the initial
data
\begin{equation}
(\rho,\textbf{v})(0,x)=(\rho_{0},\textbf{v}_{0}).\label{R-E2}
\end{equation}

Our first interest is, for fixed $\tau>0$, to investigate the
relaxation effect on the regularity and large-time behavior of
classical solutions. As shown by \cite{STW,WY}, if the initial data
are small in some Sobolev space $H^{s}(\mathbb{R}^{d})$ with
$s>1+d/2$ ($s\in\mathbb{Z}$), the relaxation term which plays the
role of damping, can prevent the development of shock waves in
finite time and the Cauchy problem (\ref{R-E1})-(\ref{R-E2}) admits
a unique global classical solution. Furthermore, it is proved that
the solution in \cite{STW} has the $L^{\infty}$ convergence rate
$(1+t)^{-3/2}(d=3)$ to the constant background state and the optimal
$L^{p}(1<p\leq\infty)$ convergence rate $(1+t)^{-d/2(1-1/p)}$ in
general several dimensions \cite{WY}, respectively. For the
one-dimensional Euler equations with relaxation, the global
existence of a smooth solution with small data was proved by Nishida
\cite{N2}, and the asymptotic behavior of the smooth solution was
studied in many papers, see e.g. the excellent survey paper by
Dafermos \cite{D1} and the book by Hsiao \cite{H}. In addition, for
the large-time behavior of solutions with vacuum, see
\cite{HP2,HMP}.

Another main interest is to justify the singular limit as
$\tau\rightarrow0$ in (\ref{R-E1}). First, we look at the formal
process. To do this, we change the time variable by considering an
``$\mathcal{O}(1/\tau)$" time scale:
\begin{equation}(\rho^{\tau},\textbf{v}^{\tau})(s,x)=\Big(\rho,\textbf{v}\Big)\Big(\frac{s}{\tau},x\Big). \label{R-E3} \end{equation}
Then the new variables satisfy the following equations:
\begin{equation}
\left\{
\begin{array}{l}\partial_{s}\rho^{\tau}+\nabla\cdot(\frac{\rho^{\tau}\textbf{v}^{\tau}}{\tau})=0,\cr
 \tau^2\partial_{s}(\frac{\rho^{\tau}\textbf{v}^{\tau}}{\tau})+\tau^2\nabla\cdot(\frac{\rho^{\tau}\textbf{v}^{\tau}\otimes\textbf{v}^{\tau}}{\tau^2})+\frac{\rho^{\tau}\textbf{v}^{\tau}}{\tau}=-\nabla
 P(\rho^{\tau})
\end{array} \right.\label{R-E4}
\end{equation}
with initial data
\begin{equation}(\rho^{\tau},\textbf{v}^{\tau})(x,0)=(\rho_{0},
\textbf{v}_{0}).\label{R-E5}\end{equation} At the formal level, if
we can show that $\frac{\rho^\tau \mathbf{v}^\tau}{\tau}$ is
uniformly bounded, it is not difficult to see that the limit
$\mathcal{N}$ of $\rho^{\tau}$ as $\tau\rightarrow0$ satisfies the
porous medium equation
\begin{equation} \left\{
\begin{array}{l}\partial_{s}\mathcal{N}-\Delta P(\mathcal{N})=0,\\ \mathcal{N}(x,0)=\rho_{0}.
\end{array} \right.\label{R-E6}
\end{equation}
which is a parabolic equation since $P(\mathcal{N})$ is strictly
increasing.

This singular limit problems for hyperbolic relaxation to parabolic
equations have attracted much attention. By means of compensated
compactness theory, Marcati and his collaborators \cite{MM,MMS,MR}
systematically studied this diffusive limit of generally
quasi-linear hyperbolic system, also including the present Euler
equations (\ref{R-E1}) for weak solutions. When $\gamma=1$, Junca
and Rascle \cite{JR} verified the convergence of solutions to the
isothermal equations (\ref{R-E1}) towards the solution to the heat
equation for arbitrarily large initial data in $BV(\mathbb{R})$ that
are bounded away from the vacuum. Coulombel and Goudon \cite{CG}
fell back on the classical energy approach and constructed the
uniform smooth solutions to the isothermal Euler equations and
justified the relaxation limit in some Sobolev space
$H^{s}(\mathbb{R}^{d})(s>1+d/2,\ s\in \mathbb{Z})$ (in $x$).

In the present paper, we will improve Coulombel and Goudon's work
\cite{CG} such that the result may hold in the critical space with
the regularity index $\sigma=1+d/2$ (a larger space). Indeed, we
choose the critical Besov space $B^{\sigma}_{2,1}(\mathbb{R}^{d})$
in space-variable $x$ rather than $H^{\sigma}(\mathbb{R}^{d})$ as
the functional setting, since $B^{\sigma}_{2,1}(\mathbb{R}^{d})$ is
a subalgebra of $\mathcal{W}^{1,\infty}$. Starting from this simple
consideration, based on the Littlewood-Paley decomposition theory
and Bony's para-product formula, we first construct the (uniform)
global existence of classical solutions in the framework of the
Chemin-Lerner's spaces $\widetilde{L}^{\theta}_{T}(B^{s}_{p,r})$ in
\cite{C2}, which is a refinement of the usual spaces
$L^{\theta}_{T}(B^{s}_{p,r})$. Then, using Aubin-Lions compactness
lemma, we perform the relaxation limit of (\ref{R-E1})-(\ref{R-E2})
in Besov spaces.

Main results are stated as follows.
\begin{thm}\label{thm1.1}
Let $\bar{\rho}>0$ be a constant reference density. Suppose that \
$\rho_{0}-\bar{\rho}$ and $\mathbf{v}_{0}\in
B^{\sigma}_{2,1}(\mathbb{R}^{d})(\sigma=1+d/2)$, there exists a
positive constant $\delta_{0}$ independent of $\tau$ such that if
$$\|(\rho_{0}-\bar{\rho},\mathbf{v}_{0})\|_{B^{\sigma}_{2,1}(\mathbb{R}^{d})}\leq
\delta_{0},$$ then the Cauchy problem (\ref{R-E1})-(\ref{R-E2}) has
a unique global solution $(\rho,\mathbf{v})$ satisfying
\begin{eqnarray*}
(\rho,\mathbf{v})\in \mathcal{C}^{1}(\mathbb{R}^{+}\times
\mathbb{R}^{d})
\end{eqnarray*}and
\begin{eqnarray*}
(\rho-\bar{\rho},\mathbf{v}) \in
\widetilde{\mathcal{C}}(B^{\sigma}_{2,1}(\mathbb{R}^{d}))\cap
\widetilde{\mathcal{C}}^1(B^{\sigma-1}_{2,1}(\mathbb{R}^{d})).
\end{eqnarray*}
Furthermore, the uniform energy inequality holds\begin{eqnarray}
&&\|(\rho-\bar{\rho},\mathbf{v})\|_{\widetilde{L}^\infty(B^{\sigma}_{2,1}(\mathbb{R}^{d}))}
\nonumber\\&&+\lambda_{0}\Big\{\Big\|\frac{1}{\sqrt{\tau}}\mathbf{v}\Big\|_{\widetilde{L}^2(B^{\sigma}_{2,1}(\mathbb{R}^{d}))}
+\Big\|\sqrt{\tau}\nabla\rho\Big\|_{\widetilde{L}^2(B^{\sigma-1}_{2,1}(\mathbb{R}^{d}))}\Big\}
\nonumber\\&\leq& C_{0}\|(\rho_{0}-\bar{\rho},
\mathbf{v}_{0})\|_{B^{\sigma}_{2,1}(\mathbb{R}^{d})}\label{R-E7}
\end{eqnarray}
where $0<\tau\leq1$, $\lambda_{0}$ and $C_{0}$ are some uniform
positive constants independent of $\tau$.
\end{thm}

\begin{rem}
In comparison with that in \cite{CG}, Theorem \ref{thm1.1} depends
on the low- and high-frequency decomposition methods rather than the
classical energy approach. As shown by ourselves \cite{FX}, the
low-frequency estimate of density for the Euler equations
(\ref{R-E1}) is absent. Then, we overcame the difficulty by using
Gagliardo-Nirenberg-Sobolev inequality (see, e.g., \cite{E}) to
obtain a global classical solution, however, the result fails to
hold in the critical Besov spaces mentioned above. To obtain the
desired result, in the current paper, we add the new context in the
proof of global existence. Indeed, some frequency-localization
estimates in Chemin-Lerner's spaces are developed, including a
crucial estimate of commutator, for details, see Proposition
\ref{prop4.1}, Proposition \ref{prop6.1} and Corollary \ref{cor6.2}.
\end{rem}

Based on Theorem \ref{thm1.1}, using the standard weak convergence
method and Aubin-Lions compactness lemma in \cite{S}, we further
obtain the relaxation limit of (\ref{R-E1})-(\ref{R-E2}) in the
larger framework of Besov spaces.
\begin{thm}\label{thm1.2} Let $(\rho,\mathbf{v})$ be the global solution of Theorem
\ref{thm1.1}. Then
$$\rho^{\tau}-\bar{\rho}\ \ \ \mbox{is uniformly bounded in}\ \ \mathcal{C}(\mathbb{R}^{+},B^{\sigma}_{2,1}(\mathbb{R}^{d}));$$
$$\frac{\rho^{\tau}\mathbf{v}^{\tau}}{\tau}\ \ \ \mbox{is uniformly bounded in}\ \ L^2(\mathbb{R}^{+},B^{\sigma}_{2,1}(\mathbb{R}^{d})).$$
Further, there exists some function $\mathcal{N}\in
\mathcal{C}(\mathbb{R}^{+},
\bar{n}+B^{\sigma}_{2,1}(\mathbb{R}^{d}))$ which is a global weak
solution of (\ref{R-E6}). For any $0<T,R<\infty$,
$\{\rho^{\tau}(s,x)\}$ strongly converges to $\mathcal{N}(s,x)$ in
 $\mathcal{C}([0,T],
(B^{\sigma-\delta}_{2,1}(B_{r}))$ as  $\tau\rightarrow0$, where
$\delta\in(0,1)$ and $B_{r}$ denotes the ball of radius $r$ in
$\mathbb{R}^{d}$. In addition, it holds that
\begin{eqnarray}
\|(\mathcal{N}(s,\cdot)-\bar{\rho}\|_{B^{\sigma}_{2,1}(\mathbb{R}^{d})}\leq
C'_{0}\|(\rho_{0}-\bar{\rho},
\mathbf{v}_{0})\|_{B^{\sigma}_{2,1}(\mathbb{R}^{d})},\ s\geq0,
\label{R-E1000}
\end{eqnarray}
where $C'_{0}>0$ is a uniform constant independent of $\tau$.
\end{thm}

\begin{rem}
Compared with that in \cite{CG}, the relaxation convergence of
classical solutions holds in the Besov spaces with relatively
\textit{lower} regularity.  To the best of our knowledge, this is
the first result for the Euler equations (\ref{R-E1}) in this
direction. Therefore, Theorem \ref{thm1.2}  gives a rigorous
description that the porous medium equation is usually regarded as
an appropriate model for compressible inviscid fluids. In addition,
let us also mention that the limit result is generalized to be true
for general adiabatic exponent $\gamma\geq1$ but not the only case
$\gamma=1$ in \cite{CG}.
\end{rem}

The paper is organized as follows. In Section 2, we briefly review
the Littlewood-Paley decomposition theory and the characterization
of Besov spaces and Chemin-Lerner's spaces. In Section 3, we
reformulate the equations (\ref{R-E1}) as a symmetric hyperbolic
form in order to obtain the effective frequency-localization
estimate and present a local existence result for classical
solutions. In Section 4, using the high- and low-frequency
decomposition methods, we deduce the frequency-localization estimate
in Chemin-Lerner's spaces, which is used to achieve the global
existence of uniform classical solutions. Section 5 is devoted to
justify the relaxation limit for the Euler equations (\ref{R-E1}).
Finally, the paper ends with an appendix, where we give the proof of
estimates of commutator.

\textbf{Notations}. Throughout this paper, $C>0$ is a generic
constant independent of $\tau$. Denote by $\mathcal{C}([0,T],X)$
(resp., $\mathcal{C}^{1}([0,T],X)$) the space of continuous (resp.,
continuously differentiable) functions on $[0,T]$ with values in a
Banach space $X$. We often label
$\|(a,b,c,d)\|_{X}=\|a\|_{X}+\|b\|_{X}+\|c\|_{X}+\|d\|_{X}$, where
$a,b,c,d\in X$. Here and below, we omit the space dependence for
simplicity, since all functional spaces are considered in
$\mathbb{R}^{d}$. Moreover, the integral $\int_{\mathbb{R}^{d}}fdx$
is labeled as $\int f$ without any ambiguity.

\section{Preliminary}
\setcounter{equation}{0}

For convenience of reader, we try to make the context
self-contained, in this section, we briefly review the
Littlewood-Paley decomposition theory and some properties of Besov
spaces and Chemin-Lerner's spaces. For more details, the reader is
referred to \cite{BCD,D}.

Let $(\varphi, \chi)$ be a couple of smooth functions valued in $[0,
1]$ such that $\varphi$ is supported in the shell
$\textbf{C}(0,\frac{3}{4},\frac{8}{3})
=\{\xi\in\mathbb{R}^{d}|\frac{3}{4}\leq|\xi|\leq\frac{8}{3}\}$,
$\chi$ is supported in the ball $\textbf{B}(0,\frac{4}{3})=
\{\xi\in\mathbb{R}^{d}||\xi|\leq\frac{4}{3}\}$ and
$$
\chi(\xi)+\sum_{q=0}^{\infty}\varphi(2^{-q}\xi)=1,\ \ \
\xi\in\mathbb{R}^{d}.
$$
Let $\mathcal{S'}$ be the dual space of the Schwartz class
$\mathcal{S}$. For $f\in\mathcal{S'}$, the nonhomogeneous dyadic
blocks are defined as follows:
$$
\Delta_{-1}f:=\chi(D)f=\tilde{\omega}\ast f\ \ \ \mbox{with}\ \
\tilde{\omega}=\mathcal{F}^{-1}\chi;
$$
$$
\Delta_{q}f:=\varphi(2^{-q}D)f=2^{qd}\int \omega(2^{q}y)f(x-y)dy\ \
\ \mbox{with}\ \ \omega=\mathcal{F}^{-1}\varphi,\ \ \mbox{if}\ \
q\geq0,
$$
where $\ast$ the convolution operator and $\mathcal{F}^{-1}$ the
inverse Fourier transform. The nonhomogeneous Littlewood-Paley
decomposition is
$$
f=\sum_{q \geq-1}\Delta_{q}f \qquad \forall f\in \mathcal{S'}.
$$
Define the low frequency cut-off by
$$S_{q}f:=\sum_{p\leq q-1}\Delta_{p}f.$$ Of course, $S_{0}f=\Delta_{-1}f$. Moreover, the above
Littlewood-Paley decomposition is almost orthogonal in $L^2$.
\begin{prop}\label{prop2.1}
For any $f\in\mathcal{S'}(\mathbb{R}^{d})$ and
$g\in\mathcal{S'}(\mathbb{R}^{d})$, the following properties hold:
$$\Delta_{p}\Delta_{q}f\equiv 0 \ \ \ \mbox{if}\ \ \ |p-q|\geq 2,$$
$$\Delta_{q}(S_{p-1}f\Delta_{p}g)\equiv 0\ \ \ \mbox{if}\ \ \ |p-q|\geq 5.$$
\end{prop}

Having defined the linear operators $\Delta_q (q\geq -1)$, we give
the definition of Besov spaces and Bony's decomposition.
\begin{defn}\label{defn2.1}
Let $1\leq p\leq\infty$ and $s\in \mathbb{R}$. For $1\leq r<\infty$,
Besov spaces $B^{s}_{p,r}\subset \mathcal{S'}$ are defined by
$$
f\in B^{s}_{p,r} \Leftrightarrow \|f\|_{B^s_{p, r}}=:
\Big(\sum_{q\geq-1}(2^{qs}\|\Delta_{q}f\|_{L^{p}})^{r}\Big)^{\frac{1}{r}}<\infty
$$
and $B^{s}_{p,\infty}\subset \mathcal{S'}$ are defined by
$$
f\in B^{s}_{p,\infty} \Leftrightarrow \|f\|_{B^s_{p, \infty}}=:
\sup_{q\geq-1}2^{qs}\|\Delta_{q}f\|_{L^{p}}<\infty.
$$
\end{defn}

\begin{defn}\label{defn2.2}
Let $f,g $ be two temperate distributions. The product $f\cdot g$
has the Bony's decomposition:
$$f\cdot g=T_{f}g+T_{g}f+R(f,g), $$
where $T_{f}g$ is paraproduct of $g$ by $f$,
$$ T_{f}g=\sum_{p\leq q-2}\Delta_{p}f\Delta_{q}g=\sum_{q}S_{q-1}f\Delta_{q}v$$
and the remainder $ R(f,g)$ is denoted by
$$R(f,g)=\sum_{q}\Delta_{q}f\tilde{\Delta}_{q}g\ \ \ \mbox{with} \ \
\tilde{\Delta}_{q}:=\Delta_{q-1}+\Delta_{q}+\Delta_{q+1}.$$
\end{defn}

As regards the remainder of paraproduct, we have the following
result.

\begin{prop} \label{prop2.2}
Let $(s_{1},s_{2})\in \mathbb{R}^2$ and $1\leq
p,p_{1},p_{2},r,r_{1},r_{2}\leq\infty$. Assume that
$$\frac{1}{p}\leq\frac{1}{p_{1}}+\frac{1}{p_{2}}\leq 1,\ \ \frac{1}{r}\leq\frac{1}{r_{1}}+\frac{1}{r_{2}},\ \ \mbox{and}\ \ s_{1}+s_{2}>0.$$
Then the remainder $R$ maps $B^{s_{1}}_{p_{1},r_{1}}\times
B^{s_{2}}_{p_{2},r_{2}}$ in
$B^{s_{1}+s_{2}+d(\frac{1}{p}-\frac{1}{p_{1}}-\frac{1}{p_{2}})}_{p,r}$
and there exists a constant $C$ such that
$$\|R(f,g)\|_{B^{s_{1}+s_{2}+d(\frac{1}{p}-\frac{1}{p_{1}}-\frac{1}{p_{2}})}_{p,r}}\leq \frac{C^{|s_{1}+s_{2}|+1}}{s_{1}+s_{2}}\|f\|_{B^{s_{1}}_{p_{1},r_{1}}}\|g\|_{B^{s_{2}}_{p_{2},r_{2}}}.$$
\end{prop}

Some conclusions will be used in subsequent analysis. The first one
is the classical Bernstein's inequality.

\begin{lem}\label{lem2.1}
Let $k\in\mathbb{N}$ and $0<R_{1}<R_{2}$. There exists a constant
$C$, depending only on $R_{1},R_{2}$ and $d$, such that for all
$1\leq a\leq b\leq\infty$ and $f\in L^{a}$,
$$
\mathrm{Supp}\ \mathcal{F}f\subset
\mathbf{B}(0,R_{1}\lambda)\Rightarrow\sup_{|\alpha|=k}\|\partial^{\alpha}f\|_{L^{b}}
\leq C^{k+1}\lambda^{k+d(\frac{1}{a}-\frac{1}{b})}\|f\|_{L^{a}};
$$
$$
\mathrm{Supp}\ \mathcal{F}f\subset
\mathbf{C}(0,R_{1}\lambda,R_{2}\lambda) \Rightarrow
C^{-k-1}\lambda^{k}\|f\|_{L^{a}}\leq
\sup_{|\alpha|=k}\|\partial^{\alpha}f\|_{L^{a}}\leq
C^{k+1}\lambda^{k}\|f\|_{L^{a}}.
$$
Here $\mathcal{F}f$ represents the Fourier transform on $f$.
\end{lem}

As a direct corollary of the above inequality, we have
\begin{rem}\label{rem2.1} For all
multi-index $\alpha$, it holds that
$$
\|\partial^\alpha f\|_{B^s_{p, r}}\leq C\|f\|_{B^{s + |\alpha|}_{p,
r}}.
$$
\end{rem}

The second one is the embedding properties in Besov spaces.
\begin{lem}\label{lem2.2} Let $s\in \mathbb{R}$ and $1\leq
p,r\leq\infty,$ then
$$B^{s}_{p,r}\hookrightarrow B^{\tilde{s}}_{p,\tilde{r}}\ \ \
\mbox{whenever}\ \ \tilde{s}<s\ \ \mbox{or}\ \ \tilde{s}=s \ \
\mbox{and}\ \ r\leq\tilde{r};$$
$$B^{s}_{p,r}\hookrightarrow B^{s-d(\frac{1}{p}-\frac{1}{\tilde{p}})}_{\tilde{p},r}\ \ \
\mbox{whenever}\ \ \tilde{p}>p;$$
$$B^{d/p}_{p,1}(1\leq p<\infty)\hookrightarrow\mathcal{C}_{0},\ \ \ B^{0}_{\infty,1}\hookrightarrow\mathcal{C}\cap L^{\infty},$$
where $\mathcal{C}_{0}$ is the space of continuous bounded functions
which decay at infinity.
\end{lem}

The third one is the compactness result for Besov spaces.
\begin{prop}\label{prop2.3}
Let $1\leq p,r\leq \infty,\ s\in \mathbb{R}$ and $\varepsilon>0$.
For all $\phi\in C_{c}^{\infty}(\mathbb{R}^{d})$, the map
$f\mapsto\phi f$ is compact from
$B^{s+\varepsilon}_{p,r}(\mathbb{R}^{d})$ to
$B^{s}_{p,r}(\mathbb{R}^{d})$.
\end{prop}

On the other hand, we also present the definition of Chemin-Lerner's
spaces first incited by J.-Y. Chemin and N. Lerner \cite{C2}, which
is the refinement of the spaces $L^{\theta}_{T}(B^{s}_{p,r})$.

\begin{defn}\label{defn2.3}
For $T>0, s\in\mathbb{R}, 1\leq r,\theta\leq\infty$, set (with the
usual convention if $r=\infty$)
$$\|f\|_{\widetilde{L}^{\theta}_{T}(B^{s}_{p,r})}:
=\Big(\sum_{q\geq-1}(2^{qs}\|\Delta_{q}f\|_{L^{\theta}_{T}(L^{p})})^{r}\Big)^{\frac{1}{r}}.$$
Then we define the space $\widetilde{L}^{\theta}_{T}(B^{s}_{p,r})$
as the completion of $\mathcal{S}$ over $(0,T)\times\mathbb{R}^{d}$
by the above norm.
\end{defn}
Furthermore, we define
$$\widetilde{\mathcal{C}}_{T}(B^{s}_{p,r}):=\widetilde{L}^{\infty}_{T}(B^{s}_{p,r})\cap\mathcal{C}([0,T],B^{s}_{p,r})
$$ and $$\widetilde{\mathcal{C}}^1_{T}(B^{s}_{p,r}):=\{f\in\mathcal{C}^1([0,T],B^{s}_{p,r})|\partial_{t}f\in\widetilde{L}^{\infty}_{T}(B^{s}_{p,r})\}.$$
The index $T$ will be omitted when $T=+\infty$. Let us emphasize
that

\begin{rem}\label{rem2.2}
\rm According to Minkowski's inequality, it holds that
$$\|f\|_{\widetilde{L}^{\theta}_{T}(B^{s}_{p,r})}\leq\|f\|_{L^{\theta}_{T}(B^{s}_{p,r})}\,\,\,
\mbox{if}\,\, r\geq\theta;\ \ \ \
\|f\|_{\widetilde{L}^{\theta}_{T}(B^{s}_{p,r})}\geq\|f\|_{L^{\theta}_{T}(B^{s}_{p,r})}\,\,\,
\mbox{if}\,\, r\leq\theta.
$$\end{rem}
Then, we state the property of continuity for product in
Chemin-Lerner's spaces $\widetilde{L}^{\theta}_{T}(B^{s}_{p,r})$.
\begin{prop}\label{prop2.4}
The following estimate holds:
$$
\|fg\|_{\widetilde{L}^{\theta}_{T}(B^{s}_{p,r})}\leq
C(\|f\|_{L^{\theta_{1}}_{T}(L^{\infty})}\|g\|_{\widetilde{L}^{\theta_{2}}_{T}(B^{s}_{p,r})}
+\|g\|_{L^{\theta_{3}}_{T}(L^{\infty})}\|f\|_{\widetilde{L}^{\theta_{4}}_{T}(B^{s}_{p,r})})
$$
whenever $s>0, 1\leq p\leq\infty,
1\leq\theta,\theta_{1},\theta_{2},\theta_{3},\theta_{4}\leq\infty$
and
$$\frac{1}{\theta}=\frac{1}{\theta_{1}}+\frac{1}{\theta_{2}}=\frac{1}{\theta_{3}}+\frac{1}{\theta_{4}}.$$
As a direct corollary, it holds that
$$\|fg\|_{\widetilde{L}^{\theta}_{T}(B^{s}_{p,r})}
\leq
C\|f\|_{\widetilde{L}^{\theta_{1}}_{T}(B^{s}_{p,r})}\|g\|_{\widetilde{L}^{\theta_{2}}_{T}(B^{s}_{p,r})}$$
whenever $s\geq d/p,
\frac{1}{\theta}=\frac{1}{\theta_{1}}+\frac{1}{\theta_{2}}.$
\end{prop}

In addition, the estimate of commutators in
$\widetilde{L}^{\theta}_{T}(B^{s}_{p,1})$ spaces is also frequently
used in the subsequent analysis. The indices $s,p$ behave just as in
the stationary case \cite{BCD,D} whereas the time exponent $\rho$
behaves according to H\"{o}lder inequality.
\begin{lem}\label{lem2.3}
Let $1\leq p\leq\infty$ and $1\leq \theta\leq\infty$, then the
following inequality is true:
\begin{eqnarray*}
2^{qs}\|[f,\Delta_{q}]\mathcal{A}g\|_{L^{\theta}_{T}(L^{p})}\leq
 Cc_{q}\|f\|_{\widetilde{L}^{\theta_{1}}_{T}(B^{s}_{p,1})}\|g\|_{\widetilde{L}^{\theta_{2}}_{T}(B^{s}_{p,1})},\
\ s=1+d/p,
\end{eqnarray*}
where the commutator $[\cdot,\cdot]$ is defined by $[f,g]=fg-gf$,
the operator $\mathcal{A}=\mathrm{div}$ or $\mathrm{\nabla}$, $C$ is
a generic constant, and $c_{q}$ denotes a sequence such that
$\|(c_{q})\|_{ {l^{1}}}\leq
1,\frac{1}{\theta}=\frac{1}{\theta_{1}}+\frac{1}{\theta_{2}}.$
\end{lem}

Finally, we state a continuity result for compositions to end up
this section.
\begin{prop}\label{prop2.4}
Let $s>0$, $1\leq p, r, \theta\leq \infty$, $F\in
W^{[s]+1,\infty}_{loc}(I;\mathbb{R})$ with $F(0)=0$, $T\in
(0,\infty]$ and $v\in \widetilde{L}^{\theta}_{T}(B^{s}_{p,r})\cap
L^{\infty}_{T}(L^{\infty}).$ Then
$$\|F(v)\|_{\widetilde{L}^{\theta}_{T}(B^{s}_{p,r})}\leq
C(1+\|v\|_{L^{\infty}_{T}(L^{\infty})})^{[s]+1}\|v\|_{\widetilde{L}^{\theta}_{T}(B^{s}_{p,r})}.$$
\end{prop}

\section{Symmetrization and local existence}
\setcounter{equation}{0}

In terms of the ideas in \cite{STW}, we introduce a new variable
(sound speed) which transforms the
 equations (\ref{R-E1}) into a symmetric hyperbolic system.
For the isentropic case $(\gamma>1)$,
 denote the sound speed by
$$\psi(\rho)=\sqrt{P'(\rho)},$$ and set $\bar{\psi}=\psi(\bar{\rho})$ corresponding to the sound speed at a background density $\bar{\rho}>0$.
Let
$$\varrho=\frac{2}{\gamma-1}(\psi(\rho)-\bar{\psi}).$$
Then the Euler equations (\ref{R-E1}) is transformed into the
symmetric form for classical solutions:
\begin{equation}
\left\{
\begin{array}{l}\partial_{t}\varrho+\bar{\psi}\mbox{div}\textbf{v}=-\textbf{v}\cdot\nabla
\varrho-\frac{\gamma-1}{2}\varrho\mbox{div}\textbf{v},\\
 \partial_{t}\textbf{v}+\bar{\psi}\nabla \varrho+\frac{1}{\tau}\textbf{v}=-\textbf{v}\cdot\nabla\textbf{v}-\frac{\gamma-1}{2}\varrho\nabla \varrho.
\end{array} \right. \label{R-E8}
\end{equation}
The initial data (\ref{R-E2}) become
\begin{equation}(\varrho,\textbf{v})|_{t=0}=(\varrho_{0},\textbf{v}_{0}) \label{R-E9}\end{equation}
with
$$\varrho_{0}=\frac{2}{\gamma-1}(\psi(\rho_{0})-\bar{\psi}).$$

\begin{rem} \label{rem3.1}
The variable change is from the open set
$\{(\rho,\textbf{v})\in (0,+\infty)\times \mathbb{R}^{d}\}$ to the
whole space $\{(\varrho,\textbf{v})\in \mathbb{R}\times
\mathbb{R}^{d}\}$. It is easy to show that for classical solutions
$(\rho,\textbf{v})$ away from vacuum, (\ref{R-E1})-(\ref{R-E2}) is
equivalent to (\ref{R-E8})-(\ref{R-E9}) with
$\frac{\gamma-1}{2}\varrho+\bar{\psi}>0$.
\end{rem}

\begin{rem}
For the isothermal case $\gamma=1$, let us introduce the enthalpy
change $\varrho(t,x)=\sqrt{A}(\ln \rho-\ln \bar{\rho})$. In this
case, the equations (\ref{R-E1}) can be transformed into the system
(\ref{R-E8}) with $\gamma=1$, the reader is referred to \cite{FX}
for more details. In subsequent sections, we focus mainly on the
case $\gamma>1$, since the isothermal case can be dealt with at a
similar manner.
\end{rem}

Recently, we have achieved a local existence theory of classical
solutions in the framework of Chemin-Lerner's spaces for
compressible Euler-Maxwell equations, see \cite{X}. Actually, the
new result is applicable to generally symmetrizable hyperbolic
systems, including the current Euler equations of special form.
Here, we present the result only and the details of the proof are
omitted for brevity.
\begin{prop}\label{prop3.1} For any fixed relaxation time $\tau>0$,
assume that $(\varrho_{0},\mathbf{v}_{0})\in{B^{\sigma}_{2,1}}$
satisfying $\frac{\gamma-1}{2}\varrho_{0}+\bar{\psi}>0$, then there
exists a time $T_{0}>0$ (depending only on the initial data) and a
unique solution $(\varrho,\mathbf{v})$ to (\ref{R-E8})-(\ref{R-E9})
such that $(\varrho,\mathbf{v})\in \mathcal{C}^{1}([0,T_{0}]\times
\mathbb{R}^{d})$ with $\frac{\gamma-1}{2}\varrho+\bar{\psi}>0$ for
all $t\in[0,T_{0}]$ and $(\varrho,\mathbf{v})\in
\widetilde{\mathcal{C}}_{T_{0}}(B^{\sigma}_{2,1})\cap
\widetilde{\mathcal{C}}^1_{T_{0}}(B^{\sigma-1}_{2,1})$.
\end{prop}

\section{Global existence}
\setcounter{equation}{0} In this section, we first establish a
crucial \textit{a priori} estimate in Chemin-Lerner's spaces. Then
by the standard boot-strap argument, we obtain the global existence
of classical solutions of (\ref{R-E8})-(\ref{R-E9}).

The  \textit{a priori} estimate is comprised in the following
proposition.
\begin{prop}\label{prop4.1}
 Let $(\varrho,\mathbf{v})\in
\widetilde{\mathcal{C}}_{T}(B^{\sigma}_{2,1})\cap
\widetilde{\mathcal{C}}^1_{T}(B^{\sigma-1}_{2,1})$ be the solution
of (\ref{R-E8})-(\ref{R-E9}) for any given time $T>0$. There exist
some positive constants $\delta_{1}, \lambda_{1}$ and $C_{1}$
independent of $\tau$ such that if
\begin{eqnarray}\|(\varrho,\mathbf{v})\|_{\widetilde{L}^\infty_{T}(B^{\sigma}_{2,1})}\leq
\delta_{1},\label{R-E10}\end{eqnarray} then
\begin{eqnarray}&&\|(\varrho,\mathbf{v})\|_{\widetilde{L}^\infty_{T}(B^{\sigma}_{2,1})}
\nonumber\\&&+\lambda_{1}\Big\{\Big\|\frac{1}{\sqrt{\tau}}\mathbf{v}\Big\|_{\widetilde{L}^2_{T}(B^{\sigma}_{2,1})}
+\Big\|\sqrt{\tau}\nabla\varrho\Big\|_{\widetilde{L}^2_{T}(B^{\sigma-1}_{2,1})}\Big\}
\nonumber\\&\leq&
C_{1}\|(\varrho_{0},\mathbf{v}_{0})\|_{B^{\sigma}_{2,1}}.\label{R-E11}
\end{eqnarray}
\end{prop}

\begin{proof}
The proof of Proposition \ref{prop4.1}, in fact, is to capture the
dissipation rates of $(\varrho,\textbf{v})$ in turn by using the
low- and high-frequency decomposition methods, so we divide it into
several steps.\\

\textbf{Step 1. The $\widetilde{L}^\infty_{T}(B^{\sigma}_{2,1})$
estimate of $(\varrho,\textbf{v})$ and the
$\widetilde{L}^2_{T}(B^{\sigma}_{2,1})$ one of $\textbf{v}$;}

Firstly, we complete the proof of step 1. Applying the localization
operator $\Delta_{q}$ to (\ref{R-E8}) yields
\begin{equation}
\left\{
\begin{array}{l}
\partial_{t}\Delta_{q}\varrho+\bar{\psi}\Delta_{q}\mbox{div}\textbf{v}+(\textbf{v}\cdot\nabla)\Delta_{q}\varrho=[\textbf{v},\Delta_{q}]\cdot\nabla
\varrho-\frac{\gamma-1}{2}\Delta_{q}(\varrho\mbox{div}\textbf{v}),\\
 \partial_{t}\Delta_{q}\textbf{v}+\bar{\psi}\Delta_{q}\nabla
\varrho+(\textbf{v}\cdot\nabla)\Delta_{q}\textbf{v}+\frac{\Delta_{q}\textbf{v}}{\tau}
=[\textbf{v},\Delta_{q}]\cdot\nabla\textbf{v}-\frac{\gamma-1}{2}\Delta_{q}(\varrho\nabla
 \varrho),
\end{array} \right. \label{R-E12}
\end{equation}
where the commutator$[\cdot,\cdot]$ is defined by $[f,g]=fg-gf.$

Multiplying the first equation of (\ref{R-E12}) by
$\Delta_{q}\varrho$ and the second one by $\Delta_{q}\textbf{v}$
respectively, then integrating them over $\mathbb{R}^{d}$, we get
\begin{eqnarray}
&&\frac{1}{2}\frac{d}{dt}\Big(\|\Delta_{q}\varrho\|^2_{L^2}+\|\Delta_{q}\textbf{v}\|^2_{L^2}\Big)+\frac{1}{\tau}\|\Delta_{q}\textbf{v}\|^2_{L^2}\nonumber\\
&=&\frac{1}{2}\int\mathrm{div}\textbf{v}(|\Delta_{q}\varrho|^2+|\Delta_{q}\textbf{v}|^2)
+\int\{[\textbf{v},\Delta_{q}]\cdot\nabla
\varrho\Delta_{q}\varrho+[\textbf{v},\Delta_{q}]\cdot\nabla\textbf{v}\Delta_{q}\textbf{v}\}
\nonumber\\&&+\frac{\gamma-1}{2}\int\Delta_{q}\varrho(\nabla
\varrho\cdot\Delta_{q}v)+\frac{\gamma-1}{2}\int[\varrho,\Delta_{q}]\nabla\varrho\cdot\Delta_{q}\textbf{v}
+\frac{\gamma-1}{2}\int[\varrho,\Delta_{q}]\mathrm{div}\textbf{v}\Delta_{q}\varrho.\label{R-E13}
\end{eqnarray}
In what follows, we first bound the low frequency part of the
quality (\ref{R-E13}). By performing integration by parts and using
H\"{o}lder- and Gagliardo-Nirenberg-Sobolev inequalities, we have
$(d\geq3)$
\begin{eqnarray}
&&\frac{1}{2}\frac{d}{dt}\Big(\|\Delta_{-1}\varrho\|^2_{L^2}+\|\Delta_{-1}\textbf{v}\|^2_{L^2}\Big)+\frac{1}{\tau}\|\Delta_{-1}\textbf{v}\|^2_{L^2}\nonumber\\
&\leq&\|\textbf{v}\|_{L^{d}}\|\Delta_{-1}\varrho\|_{L^{\frac{2d}{d-2}}}\|\Delta_{-1}\nabla
\varrho\|_{L^2}+\|\nabla
\textbf{v}\|_{L^{\infty}}\|\Delta_{-1}\textbf{v}\|^2_{L^2}
\nonumber\\&&+\|[\textbf{v},\Delta_{-1}]\cdot\nabla
\varrho\|_{L^{\frac{2d}{d+2}}}\|\Delta_{-1}\varrho\|_{L^{\frac{2d}{d-2}}}+\|[\textbf{v},\Delta_{-1}]\cdot\nabla\textbf{v}\|_{L^2}\|\Delta_{-1}\textbf{v}\|_{L^2}
\nonumber\\&&+\frac{\gamma-1}{2}\|\nabla\varrho\|_{L^{d}}\|\|\Delta_{-1}\varrho\|_{L^{\frac{2d}{d-2}}}\|\Delta_{-1}\textbf{v}\|_{L^2}
+\frac{\gamma-1}{2}\|[\varrho,\Delta_{-1}]\nabla\varrho\|_{L^2}\|\Delta_{-1}\textbf{v}\|_{L^2}\nonumber\\
&&+\frac{\gamma-1}{2}\|[\varrho,\Delta_{-1}]\mathrm{div}\textbf{v}\|_{L^{\frac{2d}{d+2}}}\|\Delta_{-1}\varrho\|_{L^{\frac{2d}{d-2}}}
\nonumber\\&\leq&\|\textbf{v}\|_{L^{d}}\|\Delta_{-1}\nabla
\varrho\|^2_{L^2}+\|\nabla
\textbf{v}\|_{L^{\infty}}\|\Delta_{-1}\textbf{v}\|^2_{L^2}
+\|[\textbf{v},\Delta_{-1}]\cdot\nabla
\varrho\|_{L^{\frac{2d}{d+2}}}\|\Delta_{-1}\nabla\varrho\|_{L^{2}}\nonumber\\&&+\|[\textbf{v},\Delta_{-1}]\cdot\nabla\textbf{v}\|_{L^2}\|\Delta_{-1}\textbf{v}\|_{L^2}
+\frac{\gamma-1}{2}\|\nabla\varrho\|_{L^{d}}\|\|\Delta_{-1}\nabla\varrho\|_{L^{2}}\|\Delta_{-1}\textbf{v}\|_{L^2}
\nonumber\\&&+\frac{\gamma-1}{2}\|[\varrho,\Delta_{-1}]\nabla\varrho\|_{L^2}\|\Delta_{-1}\textbf{v}\|_{L^2}
+\frac{\gamma-1}{2}\|[\varrho,\Delta_{-1}]\mathrm{div}\textbf{v}\|_{L^{\frac{2d}{d+2}}}\|\Delta_{-1}\nabla\varrho\|_{L^{2}}.\label{R-E14}
\end{eqnarray} Integrating
(\ref{R-E14}) with respect to $t\in[0,T]$ implies
\begin{eqnarray}
&&\frac{1}{2}\Big(\|\Delta_{-1}\varrho\|^2_{L^2}+\|\Delta_{-1}\textbf{v}\|^2_{L^2}\Big)\Big|^t_0+\frac{1}{\tau}\|\Delta_{-1}\textbf{v}\|^2_{L^2_{t}(L^2)}\nonumber\\
&\leq&\|\textbf{v}\|_{L^2_{t}(L^d)}\|\Delta_{-1}\nabla\varrho\|_{L^2_{t}(L^2)}\|\Delta_{-1}\nabla
\varrho\|_{L^\infty_{t}(L^2)}+\|\nabla
\textbf{v}\|_{L^\infty_{t}(L^{\infty})}\|\Delta_{-1}\textbf{v}\|^2_{L^2_{t}(L^2)}
\nonumber\\&&+\|[\textbf{v},\Delta_{-1}]\cdot\nabla
\varrho\|_{L^2_{t}(L^{\frac{2d}{d+2}})}\|\Delta_{-1}\nabla\varrho\|_{L^2_{t}(L^{2})}+\|[\textbf{v},\Delta_{-1}]\cdot\nabla\textbf{v}\|_{L^2_{t}(L^2)}\|\Delta_{-1}\textbf{v}\|_{L^2_{t}(L^2)}
\nonumber\\&&+\frac{\gamma-1}{2}\|\nabla\varrho\|_{L^{\infty}_{t}(L^{d})}\|\Delta_{-1}\nabla\varrho\|_{L^2_{t}(L^2)}\|\Delta_{-1}\textbf{v}\|_{L^2_{t}(L^2)}
\nonumber\\&&+\frac{\gamma-1}{2}\|[\varrho,\Delta_{-1}]\nabla\varrho\|_{L^2_{t}(L^2)}\|\Delta_{-1}\textbf{v}\|_{L^2_{t}(L^2)}
\nonumber\\&&+\frac{\gamma-1}{2}\|[\varrho,\Delta_{-1}]\mathrm{div}\textbf{v}\|_{L^2_{t}(L^{\frac{2d}{d+2}})}\|\Delta_{-1}\nabla\varrho\|_{L^2_{t}(L^2)}
.\label{R-E15}
\end{eqnarray}
Then multiplying the factor $2^{-2\sigma}$ on both sides of
(\ref{R-E15}), we can get
\begin{eqnarray}
&&\frac{1}{2}2^{-2\sigma}\Big(\|\Delta_{-1}\varrho\|^2_{L^2}+\|\Delta_{-1}\textbf{v}\|^2_{L^2}\Big)+\frac{2^{-2\sigma}}{\tau}\|\Delta_{-1}\textbf{v}\|^2_{L^2_{t}(L^2)}\nonumber\\
&\leq&\frac{2^{-2\sigma}}{2}\Big(\|\Delta_{-1}\varrho_{0}\|^2_{L^2}+\|\Delta_{-1}\textbf{v}_{0}\|^2_{L^2}\Big)
+Cc^2_{-1}\|\textbf{v}\|_{\widetilde{L}^2_{T}(B^{\sigma}_{2,1})}\|\nabla\varrho\|_{\widetilde{L}^2_{T}(B^{\sigma-1}_{2,1})}\|\nabla
\varrho\|_{\widetilde{L}^\infty_{T}(B^{\sigma-1}_{2,1})}\nonumber\\&&+Cc^2_{-1}\|\nabla
\textbf{v}\|_{\widetilde{L}^\infty_{T}(B^{\sigma-1}_{2,1})}\|\textbf{v}\|^2_{\widetilde{L}^2_{T}(B^{\sigma}_{2,1})}+Cc^2_{-1}\|\textbf{v}\|_{\widetilde{L}^\infty_{T}(B^{\sigma}_{2,1})}\|\textbf{v}\|^2_{\widetilde{L}^2_{T}(B^{\sigma-1}_{2,1})}
\nonumber\\&&+Cc^2_{-1}\|\nabla\textbf{v}\|_{\widetilde{L}^2_{T}(B^{\sigma-1}_{2,1})}\|\varrho\|_{\widetilde{L}^\infty_{T}(B^{\sigma}_{2,1})}
\|\nabla\varrho\|_{\widetilde{L}^2_{T}(B^{\sigma-1}_{2,1})}
\nonumber\\&&+Cc^2_{-1}\|\nabla\varrho\|_{\widetilde{L}^{\infty}_{T}(B^{\sigma-1}_{2,1})}\|\nabla\varrho\|_{\widetilde{L}^2_{T}(B^{\sigma-1}_{2,1})}\|\textbf{v}\|_{\widetilde{L}^2_{T}(B^{\sigma}_{2,1})}
\nonumber\\&&+Cc^2_{-1}\|\varrho\|_{\widetilde{L}^\infty_{T}(B^{\sigma}_{2,1})}\|\nabla\varrho\|_{\widetilde{L}^2_{T}(B^{\sigma-1}_{2,1})}\|\textbf{v}\|_{\widetilde{L}^2_{T}(B^{\sigma}_{2,1})}
\nonumber\\&&+Cc^2_{-1}\|\nabla\varrho\|_{\widetilde{L}^\infty_{T}(B^{\sigma-1}_{2,1})}\|\textbf{v}\|_{\widetilde{L}^2_{T}(B^{\sigma}_{2,1})}\|\nabla\varrho\|_{\widetilde{L}^2_{T}(B^{\sigma-1}_{2,1})}
,\label{R-E16}
\end{eqnarray}
where we used Remark \ref{rem2.2}, Lemma \ref{lem2.3} and Corollary
\ref{cor6.2} which will be shown in the Appendix. Here and below
$C>0$ denotes a uniform constant independent of $\tau$; $\{c_{-1}\}$
denotes some sequence which satisfies $\|(c_{-1})\|_{ {l^{1}}}\leq
1$ although each $\{c_{-1}\}$ is possibly different in
(\ref{R-E16}).

Next, we turn to estimate the high-frequency part ($q\geq0$) of the
quality (\ref{R-E13}). With the aid of Cauchy-Schwartz inequality,
we have
\begin{eqnarray}
&&\frac{1}{2}\frac{d}{dt}\Big(\|\Delta_{q}\varrho\|^2_{L^2}+\|\Delta_{q}\textbf{v}\|^2_{L^2}\Big)+\frac{1}{\tau}\|\Delta_{q}\textbf{v}\|^2_{L^2}\nonumber\\
&\leq&\frac{1}{2}\|\nabla\textbf{v}\|_{L^\infty}(\|\Delta_{q}\varrho\|^2_{L^2}+\|\Delta_{q}\textbf{v}\|^2_{L^2})
+\|[\textbf{v},\Delta_{q}]\cdot\nabla
\varrho\|_{L^2}\|\Delta_{q}\varrho\|_{L^2}\nonumber\\&&+\|[\textbf{v},\Delta_{q}]\cdot\nabla\textbf{v}\|_{L^2}\|\Delta_{q}\textbf{v}\|_{L^2}
+\frac{\gamma-1}{2}\|\nabla
\varrho\|_{L^\infty}\|\Delta_{q}\varrho\|_{L^2}\|\Delta_{q}\textbf{v}\|_{L^2}\nonumber\\&&+\frac{\gamma-1}{2}\|[\varrho,\Delta_{q}]\nabla\varrho\|_{L^2}\|\Delta_{q}\textbf{v}\|_{L^2}
+\frac{\gamma-1}{2}\|[\varrho,\Delta_{q}]\mathrm{div}\textbf{v}\|_{L^2}\|\Delta_{q}\varrho\|_{L^2}.\label{R-E17}
\end{eqnarray}
By integrating (\ref{R-E17}) with respect to $t\in[0,T]$, we arrive
at
\begin{eqnarray}
&&\frac{1}{2}\Big(\|\Delta_{q}\varrho\|^2_{L^2}+\|\Delta_{q}\textbf{v}\|^2_{L^2}\Big)\Big|^t_0+\frac{1}{\tau}\|\Delta_{q}\textbf{v}\|^2_{L^2_{t}(L^2)}\nonumber\\
&\leq&\frac{1}{2}\|\nabla\textbf{v}\|_{L^2_{t}(L^\infty)}(\|\Delta_{q}\varrho\|_{L^2_{t}(L^2)}\|\Delta_{q}\varrho\|_{L^\infty_{t}(L^2)}+\|\Delta_{q}\textbf{v}\|_{L^2_{t}(L^2)}\|\Delta_{q}\textbf{v}\|_{L^\infty_{t}(L^2)})
\nonumber\\&&+\|[\textbf{v},\Delta_{q}]\cdot\nabla
\varrho\|_{L^2_{t}(L^2)}\|\Delta_{q}\varrho\|_{L^2_{t}(L^2)}+\|[\textbf{v},\Delta_{q}]\cdot\nabla\textbf{v}\|_{L^2_{t}(L^2)}\|\Delta_{q}\textbf{v}\|_{L^2_{t}(L^2)}
\nonumber\\&&+\frac{\gamma-1}{2}\|\nabla
\varrho\|_{L^\infty_{t}(L^\infty)}\|\Delta_{q}\varrho\|_{L^2_{t}(L^2)}\|\Delta_{q}\textbf{v}\|_{L^2_{t}(L^2)}\nonumber\\&&+\frac{\gamma-1}{2}\|[\varrho,\Delta_{q}]\nabla\varrho\|_{L^2_{t}(L^2)}\|\Delta_{q}\textbf{v}\|_{L^2_{t}(L^2)}
\nonumber\\&&+\frac{\gamma-1}{2}\|[\varrho,\Delta_{q}]\mathrm{div}\textbf{v}\|_{L^2_{t}(L^2)}\|\Delta_{q}\varrho\|_{L^2_{t}(L^2)}.\label{R-E18}
\end{eqnarray}
Then multiplying the factor $2^{2q\sigma}$ on both sides of
(\ref{R-E18}) and using Lemma \ref{lem2.3}, we obtain
\begin{eqnarray}
&&\frac{1}{2}2^{2q\sigma}\Big(\|\Delta_{q}\varrho\|^2_{L^2}+\|\Delta_{q}\textbf{v}\|^2_{L^2}\Big)+\frac{2^{2q\sigma}}{\tau}\|\Delta_{q}\textbf{v}\|^2_{L^2_{t}(L^2)}\nonumber\\
&\leq&\frac{1}{2}2^{2q\sigma}\Big(\|\Delta_{q}\varrho_{0}\|^2_{L^2}+\|\Delta_{q}\textbf{v}_{0}\|^2_{L^2}\Big)
+Cc_{q}^2\|\nabla\textbf{v}\|_{\widetilde{L}^2_{T}(B^{\sigma-1}_{2,1})}(\|\nabla\varrho\|_{\widetilde{L}^2_{T}(B^{\sigma-1}_{2,1})}\|\varrho\|_{\widetilde{L}^\infty_{T}(B^{\sigma}_{2,1})}\nonumber\\&&
+\|\textbf{v}\|_{\widetilde{L}^2_{T}(B^{\sigma}_{2,1})}\|\textbf{v}\|_{\widetilde{L}^\infty_{T}(B^{\sigma}_{2,1})})
+Cc_{q}^2\|\textbf{v}\|_{\widetilde{L}^2_{T}(B^{\sigma}_{2,1})}\|\varrho\|_{\widetilde{L}^\infty_{T}(B^{\sigma}_{2,1})}\|\nabla\varrho\|_{\widetilde{L}^2_{T}(B^{\sigma-1}_{2,1})}
\nonumber\\&&+Cc_{q}^2\|\textbf{v}\|_{\widetilde{L}^\infty_{T}(B^{\sigma}_{2,1})}\|\textbf{v}\|^2_{\widetilde{L}^2_{T}(B^{\sigma}_{2,1})}
+Cc_{q}^2\|\nabla
\varrho\|_{\widetilde{L}^\infty_{T}(B^{\sigma-1}_{2,1})}\|\nabla\varrho\|_{\widetilde{L}^2_{T}(B^{\sigma-1}_{2,1})}\|\textbf{v}\|_{\widetilde{L}^2_{T}(B^{\sigma}_{2,1})}
\nonumber\\&&+Cc_{q}^2\|\nabla\varrho\|_{\widetilde{L}^\infty_{T}(B^{\sigma-1}_{2,1})}\|\nabla\varrho\|_{\widetilde{L}^2_{T}(B^{\sigma-1}_{2,1})}\|\textbf{v}\|_{\widetilde{L}^2_{T}(B^{\sigma}_{2,1})}
\nonumber\\&&+Cc_{q}^2\|\nabla\varrho\|_{\widetilde{L}^\infty_{T}(B^{\sigma-1}_{2,1})}\|\textbf{v}\|_{\widetilde{L}^2_{T}(B^{\sigma}_{2,1})}\|\nabla\varrho\|_{\widetilde{L}^2_{T}(B^{\sigma-1}_{2,1})},\label{R-E19}
\end{eqnarray}
where we have used the fact $\|\Delta_{q}\nabla
f\|_{L^2}\approx2^{q}\|\Delta_{q}f\|_{L^2}(q\geq0)$ derived by Lemma
\ref{lem2.1}. The constant $C>0$ is a uniform constant independent
of $\tau$; $\{c_{q}\}$ denotes some sequence which satisfies
$\|(c_{q})\|_{ {l^{1}}}\leq 1$ although each $\{c_{q}\}$ is possibly
different in (\ref{R-E19}).

To conclude, combining (\ref{R-E16}) with (\ref{R-E19}) gives
\begin{eqnarray}
&&\frac{1}{2}2^{2q\sigma}\Big(\|\Delta_{q}\varrho\|^2_{L^2}+\|\Delta_{q}\textbf{v}\|^2_{L^2}\Big)+\frac{2^{2q\sigma}}{\tau}\|\Delta_{q}\textbf{v}\|^2_{L^2_{t}(L^2)}\nonumber\\
&\leq&\frac{1}{2}2^{2q\sigma}\Big(\|\Delta_{q}\varrho_{0}\|^2_{L^2}+\|\Delta_{q}\textbf{v}_{0}\|^2_{L^2}\Big)
+Cc_{q}^2\|\textbf{v}\|_{\widetilde{L}^2_{T}(B^{\sigma}_{2,1})}(\|\nabla\varrho\|_{\widetilde{L}^2_{T}(B^{\sigma-1}_{2,1})}\|\varrho\|_{\widetilde{L}^\infty_{T}(B^{\sigma}_{2,1})}\nonumber\\&&
+\|\textbf{v}\|_{\widetilde{L}^2_{T}(B^{\sigma}_{2,1})}\|\textbf{v}\|_{\widetilde{L}^\infty_{T}(B^{\sigma}_{2,1})})
+Cc_{q}^2\|\textbf{v}\|_{\widetilde{L}^\infty_{T}(B^{\sigma}_{2,1})}\|\textbf{v}\|^2_{\widetilde{L}^2_{T}(B^{\sigma}_{2,1})}
\nonumber\\&&+Cc_{q}^2\|
\varrho\|_{\widetilde{L}^\infty_{T}(B^{\sigma}_{2,1})}\|\nabla\varrho\|_{\widetilde{L}^2_{T}(B^{\sigma-1}_{2,1})}\|\textbf{v}\|_{\widetilde{L}^2_{T}(B^{\sigma}_{2,1})}
\ (q\geq-1).\label{R-E20}
\end{eqnarray}
By Young's inequality, we have
\begin{eqnarray}
&&2^{q\sigma}\Big(\|\Delta_{q}\varrho\|_{L^\infty_{T}(L^2)}+\|\Delta_{q}\textbf{v}\|_{L^\infty_{T}(L^2)}\Big)+\frac{\mu_{1}2^{q\sigma}}{\sqrt{\tau}}\|\Delta_{q}\textbf{v}\|_{L^2_{T}(L^2)}\nonumber\\
&\leq&C2^{q\sigma}\Big(\|\Delta_{q}\varrho_{0}\|_{L^2}+\|\Delta_{q}\textbf{v}_{0}\|_{L^2}\Big)
+Cc_{q}\sqrt{\|\varrho\|_{\widetilde{L}^\infty_{T}(B^{\sigma}_{2,1})}}\Big(\frac{1}{\sqrt{\tau}}\|\textbf{v}\|_{\widetilde{L}^2_{T}(B^{\sigma}_{2,1})}+\sqrt{\tau}\|\nabla\varrho\|_{\widetilde{L}^2_{T}(B^{\sigma-1}_{2,1})}\Big)\nonumber\\&&
+Cc_{q}\sqrt{\|\textbf{v}\|_{\widetilde{L}^\infty_{T}(B^{\sigma}_{2,1})}}\frac{1}{\sqrt{\tau}}\|\textbf{v}\|_{\widetilde{L}^2_{T}(B^{\sigma}_{2,1})}
\ (q\geq-1),\label{R-E21}
\end{eqnarray}
where $\mu_{1}$ is a positive constant independent of $\tau$.

Summing up (\ref{R-E21}) on $q\geq-1$, we immediately get
\begin{eqnarray}
&&\|(\varrho,\textbf{v})\|_{\widetilde{L}^\infty_{T}(B^{\sigma}_{2,1})}+\frac{\mu_{1}}{\sqrt{\tau}}\|\textbf{v}\|_{\widetilde{L}^2_{T}(B^{\sigma}_{2,1})}\nonumber\\
&\leq&C\|(\varrho_{0},\textbf{v}_{0})\|_{B^{\sigma}_{2,1}}
+C\sqrt{\|\varrho\|_{\widetilde{L}^\infty_{T}(B^{\sigma}_{2,1})}}\Big(\frac{1}{\sqrt{\tau}}\|\textbf{v}\|_{\widetilde{L}^2_{T}(B^{\sigma}_{2,1})}+\sqrt{\tau}\|\nabla\varrho\|_{\widetilde{L}^2_{T}(B^{\sigma-1}_{2,1})}\Big)\nonumber\\&&
+C\sqrt{\|\textbf{v}\|_{\widetilde{L}^\infty_{T}(B^{\sigma}_{2,1})}}\frac{1}{\sqrt{\tau}}\|\textbf{v}\|_{\widetilde{L}^2_{T}(B^{\sigma}_{2,1})}.\label{R-E22}
\end{eqnarray}

\textbf{Step 2. The $\widetilde{L}^2_{T}(B^{\sigma-1}_{2,1})$
estimate of $\nabla\varrho$.}

Using the second equation of (\ref{R-E8}), we have
\begin{eqnarray}
\bar{\psi}\nabla
\varrho=-\Big(\partial_{t}\textbf{v}+\frac{1}{\tau}\textbf{v}+\textbf{v}\cdot\nabla\textbf{v}+\frac{\gamma-1}{2}\varrho\nabla
\varrho\Big)\label{R-E23}.
\end{eqnarray}
Apply the operator $\Delta_{q}$ to (\ref{R-E23}) to get
\begin{eqnarray}
&&\bar{\psi}\tau\Delta_{q}\nabla
\varrho\nonumber\\&=&-\Big(\tau\Delta_{q}\partial_{t}\textbf{v}+\Delta_{q}\textbf{v}-\tau[\textbf{v},\Delta_{q}]\nabla\textbf{v}+\tau\textbf{v}\cdot\Delta_{q}\nabla\textbf{v}
\nonumber\\&&-\frac{\gamma-1}{2}\tau[\varrho,\Delta_{q}]\nabla
\varrho+\frac{\gamma-1}{2}\tau\varrho\Delta_{q}\nabla\varrho\Big).\label{R-E24}
\end{eqnarray}
Integrating the resulting equality over $\mathbb{R}^{d}$ after
multiplying $\Delta_{q}\nabla\varrho$, we have
\begin{eqnarray}
&&\bar{\psi}\tau\|\Delta_{q}\nabla
\varrho\|^2_{L^2}\nonumber\\&=&-\int\Big(\tau\Delta_{q}\partial_{t}\textbf{v}+\Delta_{q}\textbf{v}-\tau[\textbf{v},\Delta_{q}]\nabla\textbf{v}+\tau\textbf{v}\cdot\Delta_{q}\nabla\textbf{v}
\nonumber\\&&-\frac{\gamma-1}{2}\tau[\varrho,\Delta_{q}]\nabla
\varrho+\frac{\gamma-1}{2}\tau\varrho\Delta_{q}\nabla\varrho\Big)\cdot\Delta_{q}\nabla\varrho
\label{R-E25},
\end{eqnarray}
where the first integral can be estimated as
\begin{eqnarray}
-\tau\int\Delta_{q}\partial_{t}\textbf{v}\cdot\Delta_{q}\nabla\varrho&=&\tau\int\Delta_{q}\mathrm{div}\partial_{t}\textbf{v}\Delta_{q}\varrho
\nonumber\\&=&\tau\frac{d}{dt}\int\Delta_{q}\mathrm{div}\textbf{v}\Delta_{q}\varrho-\tau\int\Delta_{q}\mathrm{div}\textbf{v}\Delta_{q}\partial_{t}\varrho
\nonumber\\&=&\tau\frac{d}{dt}\int\Delta_{q}\mathrm{div}\textbf{v}\Delta_{q}\varrho\nonumber\\&&-\tau\int\Delta_{q}\mathrm{div}\textbf{v}\Delta_{q}\Big(-\bar{\psi}\mathrm{div}\textbf{v}
-\textbf{v}\cdot\nabla\varrho-\frac{\gamma-1}{2}\varrho\mathrm{div}\textbf{v}\Big)
\nonumber\\&\leq&\tau\frac{d}{dt}\int\Delta_{q}\mathrm{div}\textbf{v}\Delta_{q}\varrho+\tau\bar{\psi}\|\mathrm{div}\textbf{v}\|^2_{L^2}
+\tau\|\textbf{v}\|_{L^\infty}\|\Delta_{q}\mathrm{div}\textbf{v}\|_{L^2}\|\Delta_{q}\nabla\varrho\|_{L^2}\nonumber\\&&+\tau\|\Delta_{q}\mathrm{div}\textbf{v}\|_{L^2}
\|[\textbf{v},\Delta_{q}]\nabla\varrho\|_{L^2}+\frac{\gamma-1}{2}\tau\|\varrho\|_{L^\infty}\|\mathrm{div}\textbf{v}\|^2_{L^2}
\nonumber\\&&+\frac{\gamma-1}{2}\tau\|\mathrm{div}\textbf{v}\|_{L^2}\|[\varrho,\Delta_{q}]\mathrm{div}\textbf{v}\|_{L^2}.\label{R-E26}
\end{eqnarray}

\begin{rem}
In the inequality (\ref{R-E26}), the information behind the mass and
momentum equations of (\ref{R-E8}) help us eventually to estimate
the term $\partial_{t}\textbf{v}$ well. Otherwise, as in \cite{FX},
we have to establish an auxiliary inequality with respect to the
variable $(\varrho_{t},\textbf{v}_{t})$ to close the {\it a priori}
estimate, which leads to the tedious proof of global existence
consequently.
\end{rem}

Together with (\ref{R-E25})-(\ref{R-E26}), we are led to the
estimate
\begin{eqnarray}
&&\bar{\psi}\tau\|\Delta_{q}\nabla \varrho\|^2_{L^2}
\nonumber\\&\leq&\tau\frac{d}{dt}\int\Delta_{q}\mathrm{div}\textbf{v}\Delta_{q}\varrho+\tau\bar{\psi}\|\mathrm{div}\textbf{v}\|^2_{L^2}
+\|\Delta_{q}\textbf{v}\|_{L^2}\|\Delta_{q}\nabla\varrho\|_{L^2}
\nonumber\\&&+\tau\|\textbf{v}\|_{L^\infty}\|\Delta_{q}\mathrm{div}\textbf{v}\|_{L^2}\|\Delta_{q}\nabla\varrho\|_{L^2}\nonumber\\&&+\tau\|\Delta_{q}\mathrm{div}\textbf{v}\|_{L^2}
\|[\textbf{v},\Delta_{q}]\nabla\varrho\|_{L^2}+\frac{\gamma-1}{2}\tau\|\varrho\|_{L^\infty}\|\Delta_{q}\mathrm{div}\textbf{v}\|^2_{L^2}
\nonumber\\&&+\frac{\gamma-1}{2}\tau\|\Delta_{q}\mathrm{div}\textbf{v}\|_{L^2}\|[\varrho,\Delta_{q}]\mathrm{div}\textbf{v}\|_{L^2}
+\tau\|\textbf{v}\|_{L^\infty}\|\Delta_{q}\nabla\textbf{v}\|_{L^2}\|\Delta_{q}\nabla\varrho\|_{L^2}
\nonumber\\&&+\tau\|[\textbf{v},\Delta_{q}]\nabla\textbf{v}\|_{L^2}\|\Delta_{q}\nabla\varrho\|_{L^2}
+\frac{\gamma-1}{2}\|\varrho\|_{L^\infty}\|\Delta_{q}\nabla\varrho\|^2_{L^2}
\nonumber\\&&+\frac{\gamma-1}{2}\tau\|[\varrho,\Delta_{q}]\nabla\varrho\|_{L^2}\|\Delta_{q}\nabla\varrho\|_{L^2}.\label{R-E27}
\end{eqnarray}
Integrating (\ref{R-E27}) in $t\in[0,T]$ gives
\begin{eqnarray}
&&\bar{\psi}\tau\|\Delta_{q}\nabla \varrho\|^2_{L^2_{t}(L^2)}
\nonumber\\&\leq&\tau\Big(\|\Delta_{q}\mathrm{div}\textbf{v}\|_{L^2}\|\Delta_{q}\varrho\|_{L^2}
+\|\Delta_{q}\mathrm{div}\textbf{v}_{0}\|_{L^2}\|\Delta_{q}\varrho_{0}\|_{L^2}\Big)+\tau\bar{\psi}\|\Delta_{q}\mathrm{div}\textbf{v}\|^2_{L^2_{T}(L^2)}
\nonumber\\&&+\|\Delta_{q}\textbf{v}\|_{L^2_{T}(L^2)}\|\Delta_{q}\nabla\varrho\|_{L^2_{T}(L^2)}
+\tau\|\textbf{v}\|_{L^\infty_{T}(L^\infty)}\|\Delta_{q}\mathrm{div}\textbf{v}\|_{L^2_{T}(L^2)}\|\Delta_{q}\nabla\varrho\|_{L^2_{T}(L^2)}\nonumber\\&&+\tau\|\Delta_{q}\mathrm{div}\textbf{v}\|_{L^2_{T}(L^2)}
\|[\textbf{v},\Delta_{q}]\nabla\varrho\|_{L^2_{T}(L^2)}+\frac{\gamma-1}{2}\tau\|\varrho\|_{L^\infty_{T}(L^\infty)}\|\Delta_{q}\mathrm{div}\textbf{v}\|^2_{L^2_{T}(L^2)}
\nonumber\\&&+\frac{\gamma-1}{2}\tau\|\Delta_{q}\mathrm{div}\textbf{v}\|_{L^2_{T}(L^2)}\|[\varrho,\Delta_{q}]\mathrm{div}\textbf{v}\|_{L^2_{T}(L^2)}
+\tau\|\textbf{v}\|_{L^\infty_{T}(L^\infty)}\|\Delta_{q}\nabla\textbf{v}\|_{L^2_{T}(L^2)}\|\Delta_{q}\nabla\varrho\|_{L^2_{T}(L^2)}
\nonumber\\&&+\tau\|[\textbf{v},\Delta_{q}]\nabla\textbf{v}\|_{L^2_{T}(L^2)}\|\Delta_{q}\nabla\varrho\|_{L^2_{T}(L^2)}
+\frac{\gamma-1}{2}\|\varrho\|_{L^\infty_{T}(L^\infty)}\|\Delta_{q}\nabla\varrho\|^2_{L^2_{T}(L^2)}
\nonumber\\&&+\frac{\gamma-1}{2}\tau\|[\varrho,\Delta_{q}]\nabla\varrho\|_{L^2_{T}(L^2)}\|\Delta_{q}\nabla\varrho\|_{L^2_{T}(L^2)}.\label{R-E28}
\end{eqnarray}
Multiply the factor $2^{2q(\sigma-1)}$ on both sides of
(\ref{R-E28}) to get
\begin{eqnarray}
&&\tau2^{2q(\sigma-1)}\|\Delta_{q}\nabla \varrho\|^2_{L^2_{t}(L^2)}
\nonumber\\&\leq&C\tau
c_{q}^2\Big(\|\mathrm{div}\textbf{v}\|_{\widetilde{L}^{\infty}_{T}(B^{\sigma-1}_{2,1})}\|\varrho\|_{\widetilde{L}^{\infty}_{T}(B^{\sigma-1}_{2,1})}
+\|\mathrm{div}\textbf{v}_{0}\|_{B^{\sigma-1}_{2,1}}\|\varrho_{0}\|_{B^{\sigma-1}_{2,1}}\Big)\nonumber\\&&+C\tau
c_{q}^2\|\mathrm{div}\textbf{v}\|^2_{\widetilde{L}^2_{T}(B^{\sigma-1}_{2,1})}
+Cc_{q}^2\|\textbf{v}\|_{\widetilde{L}^2_{T}(B^{\sigma-1}_{2,1})}\|\nabla\varrho\|_{\widetilde{L}^2_{T}(B^{\sigma-1}_{2,1})}
\nonumber\\&&+C\tau
c_{q}^2\|\textbf{v}\|_{\widetilde{L}^\infty_{T}(B^{\sigma-1}_{2,1})}\|\mathrm{div}\textbf{v}\|_{\widetilde{L}^2_{T}(B^{\sigma-1}_{2,1})}\|\nabla\varrho\|_{\widetilde{L}^2_{T}(B^{\sigma-1}_{2,1})}\nonumber\\&&
+C\tau
c_{q}^2\|\varrho\|_{\widetilde{L}^{\infty}_{T}(B^{\sigma}_{2,1})}\|\textbf{v}\|^2_{\widetilde{L}^2_{T}(B^{\sigma-1}_{2,1})}
+C\tau
c_{q}^2\|\textbf{v}\|_{\widetilde{L}^{\infty}_{T}(B^{\sigma}_{2,1})}\|\textbf{v}\|_{\widetilde{L}^2_{T}(B^{\sigma}_{2,1})}\|\nabla
m\|_{\widetilde{L}^2_{T}(B^{\sigma-1}_{2,1})}\nonumber\\&&+C\tau
c_{q}^2\|\varrho\|_{\widetilde{L}^{\infty}_{T}(B^{\sigma}_{2,1})}\|\nabla\varrho\|^2_{\widetilde{L}^2_{T}(B^{\sigma-1}_{2,1})},\label{R-E29}
\end{eqnarray}
where we used Lemma \ref{lem2.3}, and $\{c_{q}\}$ denotes some
sequence which satisfies $\|(c_{q})\|_{l^1}\leq 1$.

Then it follows from Young's inequality that
\begin{eqnarray}
&&\sqrt{\tau}2^{q(\sigma-1)}\|\Delta_{q}\nabla
\varrho\|_{L^2_{T}(L^2)}\nonumber\\&\leq&
Cc_{q}\Big(\|(\varrho,\textbf{v})\|_{\widetilde{L}^{\infty}_{T}(B^{\sigma}_{2,1})}
+\|(\varrho_{0},\textbf{v}_{0})\|_{B^{\sigma}_{2,1}}\Big)+\frac{Cc_{q}}{\sqrt{\tau}}\|
\textbf{v}\|_{\widetilde{L}^2_{T}(B^{\sigma}_{2,1})}\nonumber\\&& +
Cc_{q}\sqrt{\|\textbf{v}\|_{\widetilde{L}^\infty_{T}(B^{\sigma}_{2,1})}}
\Big(\frac{1}{\sqrt{\tau}}\|\textbf{v}\|_{\widetilde{L}^2_{T}(B^{\sigma}_{2,1})}
+\sqrt{\tau}\|\nabla\varrho\|_{\widetilde{L}^2_{T}(B^{\sigma-1}_{2,1})}\Big)\nonumber\\&&
+Cc_{q}\sqrt{\|\varrho\|_{\widetilde{L}^\infty_{T}(B^{\sigma}_{2,1})}}
\frac{1}{\sqrt{\tau}}\|\textbf{v}\|_{\widetilde{L}^2_{T}(B^{\sigma}_{2,1})}
\nonumber\\&&+Cc_{q}\sqrt{\|\varrho\|_{\widetilde{L}^\infty_{T}(B^{\sigma}_{2,1})}}
\sqrt{\tau}\|\nabla\varrho\|_{\widetilde{L}^2_{T}(B^{\sigma-1}_{2,1})},
\label{R-E30}
\end{eqnarray}
where we have used the smallness of $\tau(0<\tau\leq1)$.

Finally, summing up (\ref{R-E30}) on $q\geq-1$, we deduce that
\begin{eqnarray}
&&\sqrt{\tau}\|\nabla
\varrho\|_{\widetilde{L}^2_{T}(B^{\sigma-1}_{2,1})}\nonumber\\&\leq&
C\Big(\|(\varrho,\textbf{v})\|_{\widetilde{L}^{\infty}_{T}(B^{\sigma}_{2,1})}
+\|(\varrho_{0},\textbf{v}_{0})\|_{B^{\sigma}_{2,1}}\Big)+\frac{C}{\sqrt{\tau}}\|
\textbf{v}\|_{\widetilde{L}^2_{T}(B^{\sigma}_{2,1})}\nonumber\\&& +
C\sqrt{\|\textbf{v}\|_{\widetilde{L}^\infty_{T}(B^{\sigma}_{2,1})}}
\Big(\frac{1}{\sqrt{\tau}}\|\textbf{v}\|_{\widetilde{L}^2_{T}(B^{\sigma}_{2,1})}
+\sqrt{\tau}\|\nabla\varrho\|_{\widetilde{L}^2_{T}(B^{\sigma-1}_{2,1})}\Big)\nonumber\\&&
+C\sqrt{\|\varrho\|_{\widetilde{L}^\infty_{T}(B^{\sigma}_{2,1})}}
\frac{1}{\sqrt{\tau}}\|\textbf{v}\|_{\widetilde{L}^2_{T}(B^{\sigma}_{2,1})}
+C\sqrt{\|\varrho\|_{\widetilde{L}^\infty_{T}(B^{\sigma}_{2,1})}}
\sqrt{\tau}\|\nabla\varrho\|_{\widetilde{L}^2_{T}(B^{\sigma-1}_{2,1})}.
\label{R-E31}
\end{eqnarray}

\textbf{Step 3. Combining the above analysis.}

Combining with (\ref{R-E22}) and (\ref{R-E31}), we end up with
\begin{eqnarray}
&&\|(\varrho,\textbf{v})\|_{\widetilde{L}^\infty_{T}(B^{\sigma}_{2,1})}+\frac{\mu_{1}}{\sqrt{\tau}}\|\textbf{v}\|_{\widetilde{L}^2_{T}(B^{\sigma}_{2,1})}
+\sqrt{\tau}K\|\nabla
\varrho\|_{\widetilde{L}^2_{T}(B^{\sigma-1}_{2,1})}\nonumber\\
&\leq&C\|(\varrho_{0},\textbf{v}_{0})\|_{B^{\sigma}_{2,1}}
+C\sqrt{\|\varrho\|_{\widetilde{L}^\infty_{T}(B^{\sigma}_{2,1})}}\Big(\frac{1}{\sqrt{\tau}}\|\textbf{v}\|_{\widetilde{L}^2_{T}(B^{\sigma}_{2,1})}+\sqrt{\tau}\|\nabla\varrho\|_{\widetilde{L}^2_{T}(B^{\sigma-1}_{2,1})}\Big)\nonumber\\&&
+C\sqrt{\|\textbf{v}\|_{\widetilde{L}^\infty_{T}(B^{\sigma}_{2,1})}}\frac{1}{\sqrt{\tau}}\|\textbf{v}\|_{\widetilde{L}^2_{T}(B^{\sigma}_{2,1})}
+CK\Big(\|(\varrho,\textbf{v})\|_{\widetilde{L}^{\infty}_{T}(B^{\sigma}_{2,1})}
+\|(\varrho_{0},\textbf{v}_{0})\|_{B^{\sigma}_{2,1}}\Big)\nonumber\\&&+\frac{CK}{\sqrt{\tau}}\|
\textbf{v}\|_{\widetilde{L}^2_{T}(B^{\sigma}_{2,1})}+
CK\sqrt{\|\textbf{v}\|_{\widetilde{L}^\infty_{T}(B^{\sigma}_{2,1})}}
\Big(\frac{1}{\sqrt{\tau}}\|\textbf{v}\|_{\widetilde{L}^2_{T}(B^{\sigma}_{2,1})}
+\sqrt{\tau}\|\nabla\varrho\|_{\widetilde{L}^2_{T}(B^{\sigma-1}_{2,1})}\Big)\nonumber\\&&
+CK\sqrt{\|\varrho\|_{\widetilde{L}^\infty_{T}(B^{\sigma}_{2,1})}}
\frac{1}{\sqrt{\tau}}\|\textbf{v}\|_{\widetilde{L}^2_{T}(B^{\sigma}_{2,1})}
+CK\sqrt{\|\varrho\|_{\widetilde{L}^\infty_{T}(B^{\sigma}_{2,1})}}
\sqrt{\tau}\|\nabla\varrho\|_{\widetilde{L}^2_{T}(B^{\sigma-1}_{2,1})},\label{R-E32}
\end{eqnarray}
where $K>0$ is a uniform constant independent of $\tau$. In order to
eliminate the term
$\|(\varrho,\textbf{v})\|_{\widetilde{L}^{\infty}_{T}(B^{\sigma}_{2,1})}$
and the singular one
$\|\textbf{v}\|_{\widetilde{L}^2_{T}(B^{\sigma}_{2,1})}/\sqrt{\tau}$,
we take the constant $K$ such that
$$0<K\leq\min\Big\{\frac{1}{2C},\frac{\mu_{1}}{2C}\Big\}.$$
Furthermore, it is not difficult to obtain
\begin{eqnarray}
&&\frac{1}{2}\|(\varrho,\textbf{v})\|_{\widetilde{L}^\infty_{T}(B^{\sigma}_{2,1})}+\frac{\mu_{1}}{2\sqrt{\tau}}\|\textbf{v}\|_{\widetilde{L}^2_{T}(B^{\sigma}_{2,1})}
+\sqrt{\tau}K\|\nabla
\varrho\|_{\widetilde{L}^2_{T}(B^{\sigma-1}_{2,1})}\nonumber\\
&\leq&C\|(\varrho_{0},\textbf{v}_{0})\|_{B^{\sigma}_{2,1}}
+C\sqrt{\|\varrho\|_{\widetilde{L}^\infty_{T}(B^{\sigma}_{2,1})}}\Big(\frac{1}{\sqrt{\tau}}\|\textbf{v}\|_{\widetilde{L}^2_{T}(B^{\sigma}_{2,1})}+\sqrt{\tau}\|\nabla\varrho\|_{\widetilde{L}^2_{T}(B^{\sigma-1}_{2,1})}\Big)\nonumber\\&&
+C\sqrt{\|\textbf{v}\|_{\widetilde{L}^\infty_{T}(B^{\sigma}_{2,1})}}\frac{1}{\sqrt{\tau}}\|\textbf{v}\|_{\widetilde{L}^2_{T}(B^{\sigma}_{2,1})}
+C K\|(\varrho_{0},\textbf{v}_{0})\|_{B^{\sigma}_{2,1}}\nonumber\\&&
+CK\sqrt{\|\textbf{v}\|_{\widetilde{L}^\infty_{T}(B^{\sigma}_{2,1})}}
\Big(\frac{1}{\sqrt{\tau}}\|\textbf{v}\|_{\widetilde{L}^2_{T}(B^{\sigma}_{2,1})}
+\sqrt{\tau}\|\nabla\varrho\|_{\widetilde{L}^2_{T}(B^{\sigma-1}_{2,1})}\Big)\nonumber\\&&
+CK\sqrt{\|\varrho\|_{\widetilde{L}^\infty_{T}(B^{\sigma}_{2,1})}}
\Big(\frac{1}{\sqrt{\tau}}\|\textbf{v}\|_{\widetilde{L}^2_{T}(B^{\sigma}_{2,1})}
+\sqrt{\tau}\|\nabla\varrho\|_{\widetilde{L}^2_{T}(B^{\sigma-1}_{2,1})}\Big)
\nonumber\\
&\leq&C\|(\varrho_{0},\textbf{v}_{0})\|_{B^{\sigma}_{2,1}}+C\sqrt{\|(\varrho,\textbf{v})\|_{\widetilde{L}^\infty_{T}(B^{\sigma}_{2,1})}}
\Big(\frac{1}{\sqrt{\tau}}\|\textbf{v}\|_{\widetilde{L}^2_{T}(B^{\sigma}_{2,1})}+\sqrt{\tau}\|\nabla\varrho\|_{\widetilde{L}^2_{T}(B^{\sigma-1}_{2,1})}\Big)
\nonumber\\
&\leq&C\|(\varrho_{0},\textbf{v}_{0})\|_{B^{\sigma}_{2,1}}+C\sqrt{\delta_{1}}
\Big(\frac{1}{\sqrt{\tau}}\|\textbf{v}\|_{\widetilde{L}^2_{T}(B^{\sigma}_{2,1})}+\sqrt{\tau}\|\nabla\varrho\|_{\widetilde{L}^2_{T}(B^{\sigma-1}_{2,1})}\Big)
,\label{R-E33}
\end{eqnarray}
where we have used the \textit{a priori} assumption (\ref{R-E10}) in the last
step of (\ref{R-E33}).

Lastly, we choose the positive constant $\delta_{1}$ satisfying
$$C\sqrt{\delta_{1}}<\min\Big\{\frac{\mu_{1}}{2},K\Big\},$$ then the
desired inequality  (\ref{R-E11}) follows immediately.
\end{proof}

With the help of the standard boot-strap argument, for instance, see
\cite{MN}, Theorem \ref{thm1.1} follows from the local existence
result (Proposition \ref{prop3.1}) and \textit{a priori} estimate
(Proposition \ref{prop4.1}). Here, we give the outline of the proof.
\\

\noindent\textit{\underline{Proof of Theorem \ref{thm1.1}.}} If the
initial data satisfy $
\|(\varrho_{0},\textbf{v}_{0})\|_{B^{\sigma}_{2,1}}\leq\frac{\delta_{1}}{2}$,
by Proposition \ref{prop3.1}, then we determine a time
$T_{1}>0(T_{1}\leq T_{0})$ such that the local solutions of
(\ref{R-E8})-(\ref{R-E9}) exists in
$\widetilde{\mathcal{C}}_{T_{1}}(B^{\sigma}_{2,1})$ and
$\|(\varrho,\textbf{v})\|_{\widetilde{L}^\infty_{T_{1}}(B^{\sigma}_{2,1})}\leq\delta_{1}$.
Therefore from Proposition \ref{prop4.1} the solutions satisfy the
\textit{a priori} estimate
$\|(\varrho,\textbf{v})\|_{\widetilde{L}^\infty_{T_{1}}(B^{\sigma}_{2,1})}\leq
C_{1}\|(\varrho_{0},\textbf{v}_{0})\|_{B^{\sigma}_{2,1}}\leq\frac{\delta_{1}}{2}$
provided $
\|(\varrho_{0},\textbf{v}_{0})\|_{B^{\sigma}_{2,1}}\leq\frac{\delta_{1}}{2C_{1}}.$
Thus by Proposition \ref{prop3.1} the system
(\ref{R-E8})-(\ref{R-E9}) for $t\geq T_{1}$ with the initial data
$(\varrho,\textbf{v})(T_{1})$ has again a unique solution
$(\varrho,\textbf{v})$ satisfying
$\|(\varrho,\textbf{v})\|_{\widetilde{L}^\infty_{(T_{1},2T_{1})}(B^{\sigma}_{2,1})}\leq\delta_{1}$,
further
$\|(\varrho,\textbf{v})\|_{\widetilde{L}^\infty_{2T_{1}}(B^{\sigma}_{2,1})}\leq\delta_{1}$.
Then by Proposition \ref{prop4.1} we have
$\|(\varrho,\textbf{v})\|_{\widetilde{L}^\infty_{2T_{1}}(B^{\sigma}_{2,1})}\leq
C_{1}\|(\varrho_{0},\textbf{v}_{0})\|_{B^{\sigma}_{2,1}}\leq\frac{\delta_{1}}{2}$.
Thus we can continuous the same process for $0\leq t\leq nT_{1},
n=3,4,...$ and finally get a global solution
$(\varrho,\textbf{v})\in\widetilde{\mathcal{C}}(B^{\sigma}_{2,1})$
satisfying
\begin{eqnarray}&&\|(\varrho,\mathbf{v})\|_{\widetilde{L}^\infty(B^{\sigma}_{2,1})}
\nonumber\\&&+\lambda_{1}\Big\{\Big\|\frac{1}{\sqrt{\tau}}\mathbf{v}\Big\|_{\widetilde{L}^2(B^{\sigma}_{2,1})}
+\Big\|\sqrt{\tau}\nabla\varrho\Big\|_{\widetilde{L}^2(B^{\sigma-1}_{2,1})}\Big\}
\nonumber\\&\leq&
C_{1}\|(\varrho_{0},\mathbf{v}_{0})\|_{B^{\sigma}_{2,1}}\leq\frac{\delta_{1}}{2}.\label{R-E34}
\end{eqnarray}
The choice of $\delta_{1}$ is sufficient to ensure
$\frac{\gamma-1}{2}\varrho+\bar{\psi}>0$. Then it follows form
Remark \ref{rem3.1} that $(\rho,\textbf{v})\in
\mathcal{C}^{1}([0,\infty)\times \mathbb{R}^{d})$ is a classical
solution of (\ref{R-E1})-(\ref{R-E2}) with $\rho>0$.\ Furthermore,
we arrive at Theorem \ref{thm1.1} with
$\delta_{0}=\min(\delta_{1}/2,\delta_{1}/2C_{1})$. \hspace{100mm}
$\square$

\section{Relaxation limit}\setcounter{equation}{0}
In this section, we give the proof of Theorem \ref{thm1.2}.
\begin{proof}
From (\ref{R-E7}) and Remark \ref{rem2.2}, we deduce that quantities
$\sup_{s\geq0}\|\rho^{\tau}-\bar{\rho}\|_{B^{\sigma}_{2,1}}$ and
$$\frac{1}{\tau}\int^\infty_{0}\|\rho\textbf{v}(t)\|^2_{B^{\sigma}_{2,1}}dt=\frac{1}{\tau^2}\int^\infty_{0}\|\rho^\tau\textbf{v}^\tau(s)\|^2_{B^{\sigma}_{2,1}}ds$$
are bounded uniformly with respect to $\tau$. Therefore, the
left-hand side of (\ref{R-E4}) reads as $\tau^2\times$ the time
derivative of a quantity which is bounded in
$L^2(\mathbb{R}^{+}\times \mathbb{R}^{d})$, plus $\tau^2\times$ the
space derivative of a quantity which is bounded in
$L^1(\mathbb{R}^{+}\times \mathbb{R}^{d})$. So, this allows us to
pass to the limit $\tau\rightarrow0$ in the sense of distributions, and
we arrive at
$$-\frac{\rho^\tau\textbf{v}^\tau}{\tau}-\nabla P(\rho^\tau)\rightharpoonup 0\ \ \ \ \mbox{in}\ \ \ \mathcal{D}'(\mathbb{R}^{+}\times \mathbb{R}^{d}).$$
Inserting the weak convergence property into the first equation of
(\ref{R-E4}), we have
$$\partial_{s}\rho^\tau-\Delta P(\rho^\tau)\rightharpoonup 0\ \ \ \ \mbox{in}\ \ \ \mathcal{D}'(\mathbb{R}^{+}\times \mathbb{R}^{d})$$
as $\tau\rightarrow0$.

On the other hand, by (\ref{R-E4}), we conclude that
$\partial_{s}\rho^\tau$ is bounded in
$L^2(\mathbb{R}^{+},B^{\sigma-1}_{2,1})$. Hence, it follows from
Proposition \ref{prop2.3} and Aubin-Lions compactness lemma in
\cite{S} that there exists some function $\mathcal{N}\in
\mathcal{C}(\mathbb{R}^{+}, \bar{\rho}+B^{\sigma}_{2,1})$ such that
as $\tau\rightarrow0$, it holds that
$$
\{\rho^{\tau}\}\rightarrow \mathcal{N}\ \ \ \ \mbox{strongly in}\ \
\ \mathcal{C}([0,T],(B^{\sigma-\delta}_{2,1})(B_{r})),$$ for any
$T>0$ and $\delta\in(0,1)$, which implies that $\mathcal{N}$ is a
global weak solution to the porous medium equation (\ref{R-E6})
satisfying (\ref{R-E1000}). For more details, the reader is referred
to e.g. \cite{CG}.

Therefore, the proof of Theorem \ref{thm1.2} is complete.
\end{proof}

\section{Appendix}\setcounter{equation}{0}
As we known, Vishik, Bahouri, Chemin and Danchin \textit{et al.}
\cite{V,BCD,D} have obtained some estimates of commutator, however,
their results are unable to be applied to our case directly. Hence,
following from their arguments, we develop a new estimate of
commutator.
\begin{prop} \label{prop6.1}Let $s>0,\ 1\leq r\leq\infty$ and $p,p_{1},p_{2}\in[1,\infty]^3$ with
$1/p=1/p_{1}+1/p_{2}$. There exists a generic constant $C>0$
depending only on $p,p_{1},p_{2},\sigma, r, d$ such that
\begin{equation}
2^{qs}\|[f,\Delta_{q}]\mathcal{A}g\|_{L^{p}}\leq Cc_{q}\|\nabla
f\|_{L^{p_{1}}\cap B^{\sigma-1}_{p_{2},r}\cap
B^0_{p_{1},r}}(\|g\|_{B^{\sigma}_{p_{2},r}}+\|\nabla
g\|_{L^{p_{1}}}),\label{R-E35}
\end{equation}
where the operator $\mathcal{A}:=\mathrm{div}$ or $\nabla$. As a
direct consequence, when $1\leq p\leq p_{2}\leq p_{1}\leq\infty$, if
$$s>1+d\Big(\frac{1}{p_{2}}-\frac{1}{p_{1}}\Big)\ \ \ \mbox{or}\ \ \ s=1+d\Big(\frac{1}{p_{2}}-\frac{1}{p_{1}}\Big)\ \ \mbox{and}\ \ r=1,$$
then
\begin{equation}
2^{qs}\|[f,\Delta_{q}]\mathcal{A}g\|_{L^{p}}\leq Cc_{q}\|\nabla
f\|_{B^{s-1}_{p_{2},r}}\|g\|_{B^{s}_{p_{2},r}}, \label{R-E36}
\end{equation}
where $\{c_{q}\}$ denotes a sequence such that $\|(c_{q})\|_{
{l^{r}}}\leq 1.$
\end{prop}
\begin{proof}
To show that the gradient part of $f$ is involved in the estimate,
we need to split $f$ into low and high frequencies:
$f=\Delta_{-1}f+\tilde{f}$. Obviously, there exists a constant $C>0$
such that
\begin{eqnarray}
\|\Delta_{-1}\nabla f\|_{L^{p}}\leq C \|\nabla f\|_{L^{p}},\ \,\
\|\nabla\tilde{f}\|_{L^{p}}\leq C \|\nabla f\|_{L^{p}}, \ \
p\in[1,\infty].\label{R-E37}
\end{eqnarray}
Since $\tilde{\varrho}$ is spectrally supported away from the
origin, that is, there exists a radius $0<R<\frac{3}{4}$ such that
$\mathrm{Supp}\ \mathcal{F}\tilde{f}\bigcap B(0,R)=\emptyset$, Lemma
\ref{lem2.1} implies
\begin{eqnarray}\|\Delta_{q}\nabla\tilde{f}\|_{L^{p}}\approx
2^{q}\|\Delta_{q}\tilde{f}\|_{L^{p}}, \ \ \ p\in[1,\infty],\ \ \
q\geq-1.\label{R-E38}\end{eqnarray}

Without loss of generality, we proceed the proof with
$\mathcal{A}g=\mathrm{div}g$. Taking advantage of Bony's
decomposition, we have
\begin{eqnarray*}
[f,\Delta_{q}]\mbox{div}g&=&[\tilde{f},\Delta_{q}]\mbox{div}g+[\Delta_{-1}f,\Delta_{q}]\mbox{div}g\\
&=&\tilde{f}\Delta_{q}\mbox{div}g-\Delta_{q}(\tilde{f}\mbox{div}g)+[\Delta_{-1}f,\Delta_{q}]\mbox{div}g
\\&=&T_{\tilde{f}}\Delta_{q}\mbox{div}g+T_{\Delta_{q}\mbox{div}g}\tilde{f}+R(\tilde{f},\Delta_{q}\mbox{div}g)\\
&&-\Delta_{q}(T_{\tilde{f}}\mbox{div}g+T_{\mbox{div}g}\tilde{f}+R(\tilde{f},\mbox{div}g))+[\Delta_{-1}f,\Delta_{q}]\mbox{div}g.
\end{eqnarray*}
Set $[f,\Delta_{q}]\mbox{div}g\equiv\sum^6_{i=1}F^{i}_{q} $, where
\begin{eqnarray*}F^1_{q}&=&T_{\tilde{f}}\Delta_{q}\partial_{j}g^{j}-\Delta_{q}T_{\tilde{f}}\partial_{j}g^{j},\ \ \ \ (\mbox{div}g:=\partial_{j}g^{j})\\
F^2_{q}&=&T_{\Delta_{q}\partial_{j}g^{j}}\tilde{f},\\
F^3_{q}&=&-\Delta_{q}T_{\partial_{j}g^{j}}\tilde{f},\\
F^4_{q}&=&\partial_{j}R(\tilde{f},\Delta_{q}g^{j})-\partial_{j}\Delta_{q}R(\tilde{f},g^{j}),\\
F^5_{q}&=&\Delta_{q}R(\partial_{j}\tilde{f},g^{j})-R(\partial_{j}\tilde{f},\Delta_{q}g^{j})\\
F^6_{q}&=&[\Delta_{-1}f,\Delta_{q}]\mbox{div}g.
\end{eqnarray*}
By Proposition \ref{prop2.1}, we have
\begin{eqnarray*}
F^1_{q}&=&\sum_{q'}S_{q'-1}\tilde{f}\Delta_{q'}\Delta_{q}\partial_{j}g^{j}-\Delta_{q}\sum_{q'}S_{q'-1}\tilde{f}\Delta_{q'}\partial_{j}g^{j}\\
&=&\sum_{|q-q'|\leq4}[S_{q'-1}\tilde{f},\Delta_{q}]\partial_{j}\Delta_{q'}g^{j}\\
&=&\sum_{|q-q'|\leq4}\int_{\mathbb{R}^{d}}
h(y)[S_{q'-1}\tilde{f}(x)-S_{q'-1}\tilde{f}(x-2^{-q}y)]\partial_{j}\Delta_{q'}g^{j}(x-2^{-q}y)dy.
\end{eqnarray*}
Then, applying first order Taylor's formula, Young's inequality,
Lemma \ref{lem2.1} and (\ref{R-E37}), we get
\begin{eqnarray*}
2^{q\sigma}\|F^1_{q}\|_{L^{p}}&\leq&C\sum_{|q-q'|\leq4}\|\nabla \tilde{f}\|_{L^{p_{1}}}2^{(\sigma-1)(q-q')}2^{q'\sigma}\|\Delta_{q'}g^{j}\|_{L^{p_{2}}}\\
&\leq&\ Cc_{q1}\|\nabla
f\|_{L^{p_{1}}}\|g\|_{B^{\sigma}_{p_{2},r}},\ \ \ \
c_{q1}:=\sum_{|q-q'|\leq4}\frac{2^{q'\sigma}\|\Delta_{q'}g\|_{L^{p_{2}}}}{9\|g\|_{B^{\sigma}_{p_{2},r}}}.
\end{eqnarray*}
and
\begin{eqnarray*}
2^{q\sigma}\|F^2_{q}\|_{L^{p}}&=& 2^{q\sigma}\Big\|\sum_{q'\geq q-3}S_{q'-1}\partial_{j}\Delta_{q}g^{j}\Delta_{q'}\tilde{f}\Big\|_{L^{p}}\\
&\leq& 2^{q\sigma}\sum_{q'\geq q-3}\|\Delta_{q'}\tilde{f}\|_{L^{p_{1}}}\|S_{q'-1}\partial_{j}\Delta_{q}g^{j}\|_{L^{p_{2}}}\\
&\leq&C\sum_{q'\geq q-3}2^{q-q'}\|\nabla f\|_{L^{p_{1}}}2^{q\sigma}\|\Delta_{q}g\|_{L^{p_{2}}}\\
&\leq& Cc_{q2}\|\nabla f\|_{L^{p_{1}}}\|g\|_{B^{\sigma}_{p_{2},r}},\
\ \
c_{q2}:=\frac{2^{q\sigma}\|\Delta_{q}g\|_{L^{p_{2}}}}{\|g\|_{B^{\sigma}_{p_{2},r}}}.
\end{eqnarray*}
The third part $F^3_{q}$ is proceeded as follows:
\begin{eqnarray*}
F^3_{q}&=&-\Delta_{q}T_{\partial_{j}g^{j}}\tilde{f}\\
&=&-\sum_{|q-q'|\leq4}\Delta_{q}(S_{q'-1}\partial_{j}g^{j}\Delta_{q'}\tilde{f}),
\end{eqnarray*}
then
\begin{eqnarray*}
2^{q\sigma}\|F^3_{q}\|_{L^{p}}&\leq& C\sum_{|q-q'|\leq4}2^{(q-q')\sigma}2^{q'\sigma}\|S_{q'-1}\partial_{j}g^{j}\Delta_{q'}\tilde{f}\|_{L^p}\\
&\leq& C\sum_{|q-q'|\leq4}2^{(q-q')\sigma}\|S_{q'-1}\partial_{j}g^{j}\|_{L^{p_{1}}}2^{q'(\sigma-1)}\|\Delta_{q'}\nabla\tilde{f}\|_{L^{p_2}}\\
&\leq& Cc_{q3}\|\nabla f\|_{B^{\sigma-1}_{p_{2},r}}\|\nabla
g\|_{L^{p_{1}}},\ \ \
c_{q3}:=\sum_{|q-q'|\leq4}\frac{2^{q'(\sigma-1)}\|\Delta_{q'}\nabla
f\|_{L^{p_{2}}}}{9\|\nabla f\|_{B^{\sigma-1}_{{p_{2}},r}}}.
\end{eqnarray*}
By the definition \ref{defn2.2} and Proposition \ref{prop2.1}, we
have
\begin{eqnarray*}F^4_{q}&=&\partial_{j}R(\tilde{f},\Delta_{q}g^{j})-\partial_{j}\Delta_{q}R(\tilde{f},g^{j})\\
&=&\sum_{|q-q'|\leq1}\partial_{j}(\Delta_{q'}\tilde{f}\tilde{\Delta}_{q'}\Delta_{q}g^{j})-\partial_{j}\Delta_{q}R(\tilde{f},g^{j})\\
&=&F^{4,1}_{q}+F^{4,2}_{q}.
\end{eqnarray*}
For the first term, using (\ref{R-E38}) and Lemma \ref{lem2.1}, we
obtain
\begin{eqnarray*}
2^{q\sigma}\|F^{4,1}_{q}\|_{L^{p}}&\leq&
2^{q\sigma}\sum_{|q-q'|\leq1}\|\Delta_{q'}\nabla\tilde{f}\|_{L^{p_1}}\|\tilde{\Delta}_{q'}g^{j}\|_{L^{p_{2}}}
+2^{q\sigma}\sum_{|q-q'|\leq1}2^{q-q'}\|\Delta_{q'}\nabla\tilde{f}\|_{L^{p_{1}}}\|\tilde{\Delta}_{q'}g^{j}\|_{L^{p_{2}}}
\nonumber\\&\leq& C\|\nabla
f\|_{L^{p_{1}}}\sum_{|q-q'|\leq1}2^{(q-q')\sigma}2^{q'\sigma}\|\tilde{\Delta}_{q'}g^{j}\|_{L^{p_{2}}}\nonumber\\&&
+C\|\nabla
f\|_{L^{p_{1}}}\sum_{|q-q'|\leq1}2^{(q-q')(\sigma+1)}2^{q'\sigma}\|\tilde{\Delta}_{q'}g^{j}\|_{L^{p_{2}}}\\
&\leq& Cc_{q4(1)}\|\nabla
f\|_{L^{p_{1}}}\|g\|_{B^{\sigma}_{p_{2},r}},\ \ \
c_{4(1)}:=\sum_{|q-q'|\leq1}\frac{2^{q'\sigma}\|\Delta_{q'}g\|_{L^{p_{2}}}}{4\|g\|_{B^{\sigma}_{p_{2},r}}}.
\end{eqnarray*}
The second term is estimated as:
\begin{eqnarray*}
2^{q\sigma}\|F^{4,2}_{q}\|_{L^{p}}&=&2^{q\sigma}\|\partial_{j}\Delta_{q}R(\tilde{f},g^{j})\|_{L^{p}}\\
&\leq&C2^{q(\sigma+1)}\|\Delta_{q}R(\tilde{f},g^{j})\|_{L^{p}}\\
&\leq& Cc_{q4(2)}\|R(\tilde{f},g^{j})\|_{B^{\sigma+1}_{p,r}}\\
&\leq&
Cc_{q4(2)}\|\tilde{f}\|_{B^{1}_{p_{1},r_{1}}}\|g\|_{B^{\sigma}_{p_{2},r_{2}}}\Big(\frac{1}{r}=\frac{1}{r_{1}}+\frac{1}{r_{2}}\Big)\\
&\leq&
Cc_{q4(2)}\|\tilde{f}\|_{B^{1}_{p_{1},r}}\|g\|_{B^{\sigma}_{p_{2},r}}\\
&\leq&
Cc_{q4(2)}\|\nabla\tilde{f}\|_{B^{0}_{p_{1},r}}\|g\|_{B^{\sigma}_{p_{2},r}}\\
&\leq& Cc_{q4(2)}\|\nabla
f\|_{B^{0}_{p_{1},r}}\|g\|_{B^{\sigma}_{p_{2},r}},\ \  \
c_{q4(2)}:=\frac{2^{q(\sigma+1)}\|\Delta_{q}R(\tilde{f},g^{j})\|_{L^{p}}}{4\|R(\tilde{f},g^{j})\|_{B^{\sigma+1}_{p,r}}},
\end{eqnarray*}
where we have used Lemma \ref{lem2.2} and the result of continuity
for the remainder (Proposition \ref{prop2.2}). Among them, $s+1>0$
is required.

For $F^5_{q}$, it follows from the same argument as  $F^4_{q}$ that
$$2^{q\sigma}\|F^{5}_{q}\|_{L^{p}}\leq
Cc_{q5}\|\nabla f\|_{L^{p_{1}}\cap
B^{0}_{p_{1},r}}\|g\|_{B^{\sigma}_{p_{2},r}},$$where$$
c_{q5}:=\Big(\sum_{|q-q'|\leq1}\frac{2^{q'\sigma}\|\Delta_{q'}g\|_{L^{p_{2}}}}{4\|g\|_{B^{\sigma}_{p_{2},r}}}\Big)
+\frac{2^{q\sigma}\|\Delta_{q}R(\partial_{j}\tilde{f},g^{j})\|_{L^{p}}}{4\|R(\partial_{j}\tilde{f},g^{j})\|_{B^{\sigma}_{p,r}}},
$$
and $s>0$ is required.

For
$F^6_{q}=\sum_{|q-q'|\leq1}[\Delta_{q}(\Delta_{-1}f\partial_{j}\Delta_{q'}g^{j})-\Delta_{-1}f\Delta_{q}\Delta_{q'}\partial_{j}g^{j}]
 \ (g^{j}=\sum_{q'}\Delta_{q'}g^{j}),$ by applying first
order Taylor's formula, Young's inequality, Lemma \ref{lem2.1} and
(\ref{R-E37}), we have
\begin{eqnarray*}
2^{q\sigma}\|F^6_{q}\|_{L^{p}}&=&\Big\|\sum_{|q-q'|\leq1}\int_{\mathbb{R}^{d}}
h(y)\Big[\Delta_{-1}f(x)+\Delta_{-1}f(x-2^{-q}y)\Big]\Delta_{q'}\partial_{j}g^{j}(x-2^{-q}y)dy\Big\|_{L^{p}}\nonumber\\
&\leq& C\sum_{|q-q'|\leq1}2^{(q-q')(\sigma-1)}\|\nabla
\Delta_{-1}f\|_{L^{p_{1}}}2^{q'\sigma}\|\Delta_{q'}g\|_{L^{p_{2}}}\\
&\leq& Cc_{q6}\|\nabla f\|_{L^{p_{1}}}\|g\|_{B^{\sigma}_{p_{2},r}},\
\ \ \
c_{q6}:=\sum_{|q-q'|\leq1}\frac{2^{q'\sigma}\|\Delta_{q'}g\|_{L^{p_{2}}}}{3\|g\|_{B^{\sigma}_{p_{2},r}}}.
\end{eqnarray*}
Adding above these inequalities together, the inequality
(\ref{R-E35}) is followed with $c_{q}=\frac{1}{6}\sum_{i=1}^6c_{qi}$
satisfying $\|(c_{q})\|_{\ell^r}\leq1$.

Furthermore, if
$$s>1+d\Big(\frac{1}{p_{2}}-\frac{1}{p_{1}}\Big)\ \ \ \mbox{or}\ \ \ s=1+d\Big(\frac{1}{p_{2}}-\frac{1}{p_{1}}\Big)\ \ \mbox{and}\ \ r=1$$
with $1\leq p\leq p_{2}\leq p_{1}\leq\infty$, we have the following
embedding properties:
$$B^{s-1}_{p_{2},r}\hookrightarrow L^{p_{1}},
\ \ \ B^{s-1}_{p_{2},r}\hookrightarrow
B^{s-1-d(\frac{1}{p_{2}}-\frac{1}{p_{1}})}_{p_{1},r}\hookrightarrow
B^{0}_{p_{1},r},$$ the inequality (\ref{R-E36}) follows immediately.

Therefore, the proof of Proposition \ref{prop6.1} is complete.
\end{proof}

Having Proposition \ref{prop6.1}, we may deal with some estimates of
commutator of special form in the proof of \textit{a priori}
estimate, which are not covered by Lemma \ref{lem2.3}. For clarity,
we give them by a corollary.
\begin{cor} \label{cor6.1} Let $\sigma=1+d/2$.
There exists a generic constant $C>0$ depending only on $\sigma, d$
such that
\begin{equation}
\left\{
\begin{array}{l}2^{q\sigma}\|[\varrho,\Delta_{q}]\mathrm{div}\mathbf{v}\|_{L^{2d/d+2}}\leq
Cc_{q}\|\nabla\varrho\|_{B^{\sigma-1}_{2,1}}\|\mathbf{v}\|_{B^{\sigma}_{2,1}};\\
 2^{q\sigma}\|[\mathbf{v},\Delta_{q}]\cdot\nabla\varrho\|_{L^{2d/d+2}}\leq
Cc_{q}\|\nabla\mathbf{v}\|_{B^{\sigma-1}_{2,1}}\|\varrho\|_{B^{\sigma}_{2,1}};\\
 \end{array} \right.\label{R-E39}
\end{equation}
where $\{c_{q}\}$ denotes a sequence such that $\|(c_{q})\|_{
{l^{1}}}\leq 1.$
\end{cor}

\begin{proof} In Proposition \ref{prop6.1}, it suffices to take
$$p=\frac{2d}{d+2}(d\geq2),\ p_{2}=2,\  \sigma=1+\frac{d}{2}\ \  \mbox{and}\ \ r=1,$$
the conclusions follow obviously.
\end{proof}

According to H\"{o}lder inequality and Remark \ref{rem2.2}, it is
not difficult to achieve the estimates of commutators in
$L^{r}_{T}(L^{2d/d+2})$ spaces.

\begin{cor}\label{cor6.2}
Let $\sigma=1+d/2$ and $1\leq \theta\leq\infty$. Then there exists a
generic constant $C>0$ depending only on $\sigma, d$ such that
\begin{equation}
\left\{
\begin{array}{l}2^{q\sigma}\|[\varrho,\Delta_{q}]\mathrm{div}\mathbf{v}\|_{L^{\theta}_{T}(L^{2d/d+2})}\leq
Cc_{q}\|\nabla\varrho\|_{\widetilde{L}^{\theta_{1}}_{T}(B^{\sigma-1}_{2,1})}\|\mathbf{v}\|_{\widetilde{L}^{\theta_{2}}_{T}(B^{\sigma}_{2,1})};\\
 2^{q\sigma}\|[\mathbf{v},\Delta_{q}]\cdot\nabla\varrho\|_{L^{\theta}_{T}(L^{2d/d+2})}\leq
Cc_{q}\|\nabla\mathbf{v}\|_{\widetilde{L}^{\theta_{1}}_{T}(B^{\sigma-1}_{2,1})}\|\varrho\|_{\widetilde{L}^{\theta_{2}}_{T}(B^{\sigma}_{2,1})};\\
 \end{array} \right.\label{R-E40}
\end{equation}
where $\{c_{q}\}$ denotes a sequence such that $\|(c_{q})\|_{
{l^{1}}}\leq 1$ and
$\frac{1}{\theta}=\frac{1}{\theta_{1}}+\frac{1}{\theta_{2}}$.
\end{cor}

\begin{rem}
Actually, if we take $p=p_{2}, \ p_{1}=\infty,\ \sigma=1+d/p$\ \ and
$r=1$ in Proposition \ref{prop6.1}, we can also deduce Lemma
\ref{lem2.3}.
\end{rem}

\section*{Acknowledgments}
The research of Jiang Xu is partially supported by the NSFC
(11001127), China Postdoctoral Science Foundation (20110490134) and
NUAA Research Funding (NS2010204). The research of Zejun Wang is
partially supported by the NSFC(10901082) and China Postdoctoral
Science Foundation (20090450149).


\begin{thebibliography}{99}
\bibitem {BCD}
H. Bahouri, J.~Y. Chemin and R. Danchin.\textit{ Fourier Analysis
and Nonlinear Partial Differential Equations}, Berlin, Heidelberg:
Springer-Verlag, 2011.

\bibitem {C2}
J.~Y. Chemin, Th\'{e}or\`{e}mes d'unicit\'{e} pour le syst\`{e}me de
Navier-Stokes tridimensionnel, \textit{Journal d'Analyse
Math\'{e}matique} {\bf{77}} (1999) 25--50.

\bibitem {CG} J.-F. Coulombel and T. Goudon, The strong relaxation
limit of the multidimensional isothermal Euler equations,
\textit{Trans. Amer. Math. Soc.} {\bf{359}} (2007) 637--648.

\bibitem {D1}
C. M. Dafermos. Can dissipation prevent the breaking of waves? In:
Transactions of the Twenty-Sixth Conference of Army Mathematicians,
187¨C198, ARO Rep. 81, 1, U. S. Army Res. Office, Research Triangle
Park, N.C., 1981.

\bibitem {D}
R. Danchin, \textit{Fourier Analysis Methods for PDE's} (Lecture
Notes), 2005.

\bibitem {E}
L. C. Evans, \textit{Partial differential equations}, Providence,
Rhode Island: Amer Mathematical Society, 1998.

\bibitem {FX}
D. Y. Fang and J. Xu, Existence and asymptotic behavior of
$\mathcal{C}^1$ solutions to the multidimensional compressible Euler
equations with damping, \textit{Nonlinear Anal. TMA} {\bf{70}}
(2009) 244--261.

\bibitem {H}
L. Hsiao. \textit{Quasilinear Hyperbolic Systems and Dissipative
Mechanisms}, Singapore: World Scientific Publishing, 1997.

\bibitem {HP2} F. Huang and R. Pan, Convergence rate for compressible Euler equations with damping
and vacuum, \textit{Arch. Rational Mech. Anal.} {\bf{166}} (2003)
359--376.

\bibitem {HMP} F. Huang, P. Marcati and R. Pan,
Convergence to Barenblatt solution for the compressible Euler equations
with damping and vacuum, \textit{Arch.
Rational Mech. Anal.} {\bf{176}} (2005) 1--24.

\bibitem {JR} S. Junca and M. Rascle, Strong relaxation of the
isothermal Euler system to the heat equation, \textit{Z. Angew.
Math. Phys.} {\bf{53}} (2002) 239--264.

\bibitem{M}
A. Majda, \textit{Compressible Fluid Flow and Conservation Laws in
Several Space Variables}, Berlin, New York: Springer-Verlag, 1984.


\bibitem {MM} P. Marcati and A. Milani, The one-dimensional Darcy's
law as the limit of a compressible Euler flow, \textit{J.
Differential Equations} {\bf{84}} (1990) 129--147.

\bibitem {MMS} P. Marcati, A. Milani and P. Secchi, Singular
convergence of weak solutions for a quasilinear nonhomogeneous
hyperbolic system, \textit{Manuscripta Math.}, {\bf{60}} (1988)
49-69.

\bibitem {MN} A. Matsumura and T. Nishida, The initial value problem
for the quations  of motion of viscous and heat-conductive gases,
\textit{J. Math. Kyoto Univ.} {\bf{20}} (1980) 67--104.


\bibitem {MR} P. Marcati and B. Rubino, Hyperbolic to parabolic
relaxation theory for quasilinear first order systems, \textit{J.
Differential Equations} {\bf{162}} (2000) 359--399.


\bibitem {N2} T. Nishida, Nonlinear hyperbolic equations and
relates topics in fluid dynamics, \textit{Publ. Math. D'Orsay}
(1978) 46--53.

\bibitem{S} J.~Simon, Compact sets in the space
$L^{p}(0,T;B)$, \textit{Ann. Math. Pura Appl.}, {\bf{146}} (1987)
65--96.

\bibitem {STW}
T. Sideris, B. Thomases and  D. H. Wang, Long time behavior of
solutions to the 3D compressible Euler with damping, \textit{Comm.
P. D. E.} \textbf{28} (2003) 953--978.

\bibitem{V}
M. Vishik, Hydrodynamics in Besov spaces, \textit{Arch. Rational
Mech. Anal.} \textbf{145} (1998) 197--214.

\bibitem{WY}
W. Wang and T. Yang,  The pointwise estimates of solutions for Euler
equations with damping in multi-dimensions. \textit{J Differential
Equations} {\bf{173}} (2001) 410--450.

\bibitem{X}
J. Xu, Global classical solutions to the compressible Euler-Maxwell
equations. {\it SIAM J. Math. Anal.}, 2011, in press.



\end{thebibliography}
\end{document}